\documentclass{amsart}

\usepackage{amsmath}
\usepackage{amssymb}
\usepackage{hyperref}
\usepackage{enumitem}

\newtheorem{theorem}{Theorem}[section]
\newtheorem{lemma}[theorem]{Lemma}

\newtheorem{proposition}[theorem]{Proposition}

\theoremstyle{definition}
\newtheorem{definition}[theorem]{Definition}

\newtheorem{remark}[theorem]{Remark}

\allowdisplaybreaks

\DeclareMathOperator{\tr}{tr}

\newcommand{\Z}{\mathbb{Z}}
\newcommand{\R}{\mathbb{R}}

\newcommand{\T}{\mathbb{T}}

\let\Re=\undefined\DeclareMathOperator*{\Re}{Re}

\newcommand{\qtq}[1]{\quad\text{#1}\quad}

\let\Re=\undefined\DeclareMathOperator{\Re}{Re}

\numberwithin{equation}{section}
\numberwithin{theorem}{section}

\begin{document}

\title[Global well-posedness for mKdV in modulation spaces]{Global well-posedness and equicontinuity\\for mKdV in modulation spaces}

\author{Saikatul Haque}
\address{
Department of Mathematics\\
University of California\\Los Angeles\\CA 90095\\USA\\
\& Harish-Chandra Research Institute, Allahabad 211019, India}
\email{saikatul@math.ucla.edu}

\author{Rowan Killip}
\address{
Department of Mathematics\\
University of California\\Los Angeles\\CA 90095\\USA}
\email{killip@math.ucla.edu}

\author{Monica Vi\c{s}an}
\address{
Department of Mathematics\\
University of California\\Los Angeles\\CA 90095\\USA}
\email{visan@math.ucla.edu}

\author{Yunfeng Zhang}
\address{Department of Mathematical Sciences \\
University of Cincinnati, OH 45221, USA}
\email{zhang8y7@ucmail.uc.edu}

\begin{abstract}
We establish global well-posedness for both the defocusing and focusing complex-valued modified Korteweg--de Vries equations on the real line in modulation spaces $M_p^{s,2}(\mathbb{R})$, for all $1\leq p<\infty$ and $0\leq s<3/2-1/p$.  We will also show that such solutions admit global-in-time bounds in these spaces and that equicontinuous sets of initial data lead to equicontinuous ensembles of orbits.  Indeed, such information forms a crucial part of our well-posedness argument.
% Our arguments rely on two important ingredients: (1) the demonstration of a priori bounds in these spaces for Schwartz solutions to both mKdV and NLS; and (2) a proof that the set of orbits emanating from a bounded and equicontinuous set in $M_p^{s,2}(\mathbb{R})$ is also bounded and equicontinuous in $M_p^{s,2}(\mathbb{R})$.
\end{abstract}

%\date{\today}  % delete this later

\maketitle

%\tableofcontents

\section{Introduction}

In this paper we study the complex-valued modified Korteweg--de Vries equations
\begin{equation}\tag{mKdV}\label{mkdv}
    \partial_t u+\partial_x^3u=\pm6|u|^2\partial_xu,
\end{equation}
which will also lead us to consider the cubic nonlinear Schr\"odinger equations
\begin{equation}\tag{NLS}\label{nls}
i\partial_t u+\partial_x^2u=\pm 2|u|^2u.
\end{equation}
Here, the $+$ signs correspond to the defocusing models, while the focusing models have $-$ signs.  We are interested in studying these equations for initial data in the modulation spaces $M^{s,2}_p(\mathbb{R})$ with $1\leq p<\infty$ and $s\geq 0$. These spaces are defined as the completion of Schwartz functions with respect to the norms
\begin{equation}\label{E:mod}
\|f\|_{M^{s,2}_p}:= \bigl\| \langle k\rangle^s \|\widehat f\|_{L^2(I_k)}\bigr\|_{\ell^p_k(\Z)}.
\end{equation}
Here $\widehat f $ denotes the Fourier transform of $f$ as defined by \eqref{FT} and for each $k\in \Z$, $I_k$ denotes the interval $ [k-\frac12, k+\frac12)$.

Modulation spaces were first introduced by Feichtinger in \cite{Feichtinger83}, already in the full three-parameter generality presented in \eqref{E:Mrsp}. For a thorough introduction to this family of spaces, see \cite{Grochenig01, WHHG11}. 

%The study of dispersive equations in modulation spaces has received a great deal of attention; see \cite{WZG06, WHuang07, WH07, BO09, WHHG11, BR16, CHKP17, SGuo17, CHKP19, Pattakos19, CG20, OW20, OW21, Schippa22, Schippa23, Kla23}. These spaces have proved to be a useful alternative to the Sobolev spaces $H^s$, particularly for problems with random initial data such as that corresponding to the Gibbs state (\cite{,,}).  One of the virtues of the modulation spaces in this setting is that they more faithfully capture the distribution of norm across Fourier modes. 

The goal of this article is to establish both a priori bounds and global well-posedness for \eqref{mkdv} in modulation spaces.  Our method applies equally well to \eqref{nls} and so we will present results for this model as well.   We begin with the question of a priori bounds:

\begin{theorem}\label{T:main}
Suppose $1\leq p<\infty$ and $0\leq s<\frac32-\frac1p$. Let $u_0$ be a Schwartz function and let $u(t)$ denote the solution to either \eqref{mkdv} or \eqref{nls} with initial data $u(0)=u_0$.  Then $u$ satisfies the following a priori bounds:
\begin{align}\label{apriori bds}
\sup_{t\in\R}\, \|u(t)\|_{M^{s,2}_p}\lesssim \bigl(1+\|u_0\|_{M^{s,2}_p}\bigr)^{c(s,p)}\|u_0\|_{M^{s,2}_p}
\end{align}
where 
$$
c(s,p)=\begin{cases}
ps+\frac p2-1&\text{ if }\ 2\leq p<\infty,\\[2mm]
2s+ \frac2p-1&\text{ if }\ 1\leq p\leq 2.
\end{cases}
$$
Moreover, if $Q\subset \mathcal S(\mathbb\R)$ is bounded and equicontinuous in $M^{s,2}_p(\R)$, then the union of all orbits emanating from $Q$ under the \eqref{mkdv} or \eqref{nls} flow forms a bounded and equicontinuous set in $M^{s,2}_p(\R)$.
\end{theorem}

Theorem~\ref{T:main} expands the parameter range for a priori bounds from the earlier papers \cite{Kla23,OW20} that we will discuss more fully in subsection~\ref{SS:prior}.  The inclusion of equicontinuity is new in all cases.  The proper interpretation of this notion in the modulation spaces $M^{s,2}_p(\R)$ is presented in Definition~\ref{D:equi}.

It has long been appreciated that conservation laws play a central role in transferring \emph{local} well-posedness to \emph{global} well-posedness because they provide a priori bounds for solutions. We will see  instances of this when we review the history of well-posedness for \eqref{nls} and \eqref{mkdv} in modulation spaces.

The main theme of this paper is that the additional step of propagating \emph{equicontinuity} through time (rather than just boundedness) has an even more profound effect: it allows one to transfer well-posedness from one class of spaces to another!  Specifically, using earlier results of \cite{HKV}, we will obtain the following well-posedness results in modulation spaces:

\begin{theorem}\label{T:Main}
Suppose $1\leq p<\infty$ and $0\leq s<\frac32-\frac1p$. Then both \eqref{nls} and \eqref{mkdv} are globally well-posed in $M^{s,2}_p(\mathbb{R})$ in the following sense: The solution map extends uniquely from Schwartz space to a jointly continuous map $\Phi: \R\times M^{s,2}_p(\mathbb{R})\to M^{s,2}_p(\mathbb{R})$. Moreover, for each initial data $u_0\in M_p^{s,2}(\mathbb{R})$, the orbit $\{\Phi(t,u_0):t\in\mathbb{R}\}$ is uniformly bounded and equicontinuous in $M_p^{s,2}(\mathbb{R})$. 
\end{theorem}

In the \eqref{nls} setting, this provides a new approach to well-posedness in a regime that is accessible by prior methods; see subsection~\ref{SS:prior}.  By contrast, previous authors have not been able to treat \eqref{mkdv} for $s\leq \frac14$.  There is a reason for that: in this regime, the \eqref{mkdv} data-to-solution map is no longer real analytic.

For \eqref{mkdv}, real-valued initial data leads to real-valued solutions.  Indeed, in this case one may write the evolution equations in the form
\begin{equation}\tag{mKdV$_\R$}\label{mkdvR}
    \partial_t u+\partial_x^3u=\pm 2\partial_x (u^3).
\end{equation}
Much prior work has studied this real-valued problem by itself.  Evidently, our well-posedness result applies equally well to this special class of solutions.
On the other hand, recent work of Chapouto \cite{MR4251838} shows that the complex-valued case contains new instabilities beyond those of the real-valued case.

\subsection{Prior work}\label{SS:prior}  Historically, the Sobolev spaces $H^s(\R)=W^{s,2}(\R)=M^{s,2}_2(\R)$ have been the primary focus of research on both well-posedness and a priori bounds for dispersive equations, including \eqref{nls} and \eqref{mkdv}.

Global well-posedness of \eqref{nls} in $H^s(\R)$ for $s\geq 0$ was proved long ago by Tsutsumi \cite{MR915266}; the global nature of this result relies crucially on the exact conservation of the $L^2$ norm.

For \eqref{mkdvR}, local well-posedness in $H^s(\R)$ for $s\geq\frac14$ was proved in \cite{KPV93}.  This was made global in time in \cite{MR1969209,Guo,Kishimoto} through the construction of \emph{almost} conserved quantities.  It was only later in \cite{KVZ18,KT18} that exact conservation laws were constructed that provide uniform-in-time global bounds in $H^s(\R)$ spaces for such $s$.

The papers just mentioned obtained a very strong form of well-posedness: the solution is a real-analytic function of the initial data.  (This is a characteristic of contraction-mapping methods that transcends the complexity of the harmonic analysis tools employed.)   In this sense, the regularity thresholds presented are sharp: the data-to-solution map is no longer even uniformly continuous on bounded subsets of $H^s(\R)$ for lesser $s\in\R$; see \cite{CCT03,KPV01}.

We now know that both \eqref{nls} and \eqref{mkdv} are globally well-posed in $H^s(\R)$ for all $s>-\frac12$; this is the principal result of \cite{HKV}.   As reviewed in the appendix to that paper, this is sharp because there is instantaneous norm inflation for $s\leq-\tfrac12$.  To summarize, the well-posedness threshold for \eqref{nls} and \eqref{mkdv} in $H^s(\R)$ spaces is the scaling-critical index $s=-\tfrac12$, with ill-posedness at the borderline.

Modulation spaces have been a central character in the long-running quest to understand \eqref{nls} and \eqref{mkdv} near scaling criticality. As we will see, working in such spaces has led to some major successes, particularly, in the case of \eqref{nls}.  Before we proceed to this discussion, let us first define the full three-parameter family of modulation spaces.  To do so, we must fix a partition of unity in Fourier space.  To this end, let us choose $\widehat\phi\in C^\infty_c$ so that
\begin{equation*}%\label{E:part1}
\sum_{k\in\Z} \widehat{\phi_k}(\xi) \equiv 1 \qtq{where} \widehat{\phi_k}(\xi) := \widehat\phi(\xi-k) \qtq{or, equivalently,} \phi_k(x) := e^{ikx} \phi(x) .
\end{equation*}
The norm on $M^{s,r}_p$ is then defined via
\begin{equation}\label{E:Mrsp}
\| f \|_{M^{s,r}_p} = \bigl\| \langle k\rangle^s \| \phi_k * f\|_{L^r(\R)}\bigr\|_{\ell^p_k(\Z)}  .
\end{equation}
Different choices of the partition of unity lead to equivalent norms; indeed, when $1<r<\infty$, one may even choose $\widehat\phi$ to be the indicator of a unit interval as in \eqref{E:mod}.  Our arrangement of the indices in $M^{s,r}_p$ is designed to mimic that in $W^{s,r}$, $H^{s,r}$, and $B^{s,r}_p$; this differs from some of the references given.

While we will focus on modulation spaces in this paper, there is also a well-developed well-posedness theory in the allied setting of Fourier--Lebesgue spaces. 
In particular, we draw attention to \cite{MR2181058,MR2529909,MR2944731} and the references therein.  
%In addition to the real-line setting discussed in this paper, the well-posedness of \eqref{mkdv} and \eqref{nls} is much studied on the torus $\R/2\pi\Z$.  In this setting, there is exactly one frequency $\xi\in\Z$ per unit interval; correspondingly, we see that the natural analogue of ... is ... 

A great deal of effort has been expended in analyzing the well-posedness of nonlinear Schr\"odinger equations in the modulation spaces $M^{s,r}_p(\R)$; this includes \cite{BO09,BC20,BR16,CHKP17,CHKP19,SGuo17,Kla23,OW20,Pattakos19,Schippa22,Schippa23,WHuang07,WH07,WZG06}.  Given the diversity of methods and of parts of the parameter space that have been analyzed, we cannot give a simple cogent description of all this work.
With this in mind, we will focus mainly on the spaces $M^{s,2}_p(\R)$.
%Our primary goals are to explain why the analogue of Theorem~\ref{T:Main} is already known for \eqref{nls} and what these methods tell us about \eqref{mkdv}.
%One common attribute of all these analyses is the requirement $s\geq 0$.

We begin with local well-posedness results for \eqref{nls}. The paper \cite{WZG06} proved that \eqref{nls} is locally well-posed in  $M^{0,2}_1(\R)$.  This was expanded to $M^{s,r}_1$ with $1\leq r\leq\infty$ and $s\geq 0$ in \cite{BO09}.    These spaces are algebras.  Through the paper \cite{BR16}, we now know local well-posedness of \eqref{nls} in \emph{all} cases where $M^{s,r}_p(\R)$ is an algebra; this adds $M^{s,r}_p(\R)$ with $s>1/p'$.

In \cite{SGuo17}, Guo proved that \eqref{nls} is locally well-posed in $M^{0,2}_p$ for all $2\leq p < \infty$.  As $p\to\infty$, this result comes arbitrarily close to scaling; however, as discussed in \cite{SGuo17}, prior work shows that there is \emph{no} continuous extension of the data to solution map to the limiting case $M^{0,2}_\infty$. In \cite{BC20}, it is shown that \eqref{nls} is also ill-posed in $M^{s,2}_p$ when $s< -(\frac12 \wedge \frac1p)$.

The spaces $M^{s,2}_p$ in the range $1< p<2$ were first explored in \cite{Pattakos19} where local well-posedness and unconditional uniqueness were shown for $1\leq p<3/2$ and $s\geq 0$, as well as for $3/2<p\leq 2$ and $s\geq \tfrac32-\tfrac1p$.  Local well-posedness for $s=0$ and $1\le p<2$ was shown by contraction mapping in \cite{Kla23}; this was extended to global well-posedness in the same paper using conservation laws, as we discuss next.

The first development of conservation laws that control $M^{s,2}_p$ norms was by Oh--Wang in \cite{OW20}.  Concretely, they proved bounds of the type \eqref{apriori bds} for $2\le p<\infty$ and $0\leq s<1-\frac1p$.  Combining this with the local well-posedness results discussed earlier, they obtained global well-posedness of \eqref{nls} in these spaces.  Building on their ideas, \cite{Kla23} proved \eqref{apriori bds} for $1\le p<2$ with $s=0$, thereby obtaining global well-posedness for this range of parameters. 

Both \cite{Kla23,OW20} build on the earlier paper \cite{KVZ18}, which established a priori bounds for solutions to the Korteweg--de Vries equation, \eqref{mkdv}, and \eqref{nls}, in both $H^{s}(\R)$ and $B^{s,2}_p(\R)$ spaces.  The central object of this paper is the (logarithm of the) perturbation determinant associated with the Lax operator.  The treatment of \eqref{nls} and \eqref{mkdv} is unified by the fact that they have the same Lax operator; indeed, they belong to the same integrable hierarchy.

To prove their bounds for \eqref{nls} in modulation spaces, Oh--Wang combined this technology with the Galilei symmetry \eqref{Galilei nls} of \eqref{nls}.  The Galilei symmetry \eqref{Galilei mkdv} associated to \eqref{mkdv} is rather more complicated; nevertheless, Oh--Wang were able to prove \eqref{apriori bds} for \eqref{mkdv} over the same parameter range as for \eqref{nls}, namely,  $2\le p<\infty$ and $0\leq s<1-\frac1p$.  

We now turn to discussing the existing well-posedness theory for \eqref{mkdv} in modulation spaces.  The best prior results in this direction appeared in the contemporaneous papers \cite{CG20,OW21}. The paper \cite{CG20} shows local well-posedness in $M^{1/4,2}_p$ for $2\leq p \leq\infty$. The main result of \cite{OW21} is global well-posedness in $M^{s,2}_p$ for $s\geq \frac14$ and $2\leq p <\infty$; the authors also sketch proofs of global well-posedness for $s\geq \frac14$ and $1\leq p<2$, as well as for $s> \frac14$ and $p=\infty$.

The papers \cite{CG20,OW21} also prove mild ill-posedness for $s<\tfrac14$, specifically, that the data-to-solution map is not $C^3$ in the defocusing case and not uniformly continuous on bounded sets in the focusing case.  Thus, the regularity threshold for analytic well-posedness is $s=\frac14$ not only in $H^s(\R)$ spaces, as discussed earlier, but also in $M^{s,2}_p(\R)$, independendent of $p$.  This also highlights an important facet of our Theorem~\ref{T:Main}: we obtain well-posedness outside the realm of analytic well-posedness and correspondingly, outside the reach of contraction-mapping arguments.

\subsection{Methods}

Our proof of the a priori bounds \eqref{apriori bds} relies on the prior works \cite{KVZ18,Kla23,OW20} discussed earlier.  Following \cite{KVZ18}, we introduce the logarithm of the perturbation determinant $\alpha(\varkappa;u)$ directly as a power series in $u$; see \eqref{alpha def}. This series is shown to converge geometrically (for large values of the spectral parameter $\varkappa>0$) in Lemmas~\ref{L:HS bdd} and~\ref{L:M controls HS}.
In this way, it is natural to conflate $\alpha(\varkappa;u)$ with the leading term in the series, $\alpha_2(\varkappa;u)$, which is quadratic in $u$.

As we see from \eqref{alpha2}, this leading term captures the $L^2$ norm at low frequencies.  By combining this with the Galilei symmetry in a manner pioneered in \cite{OW20}, we deduce conservation of the $L^2$ norm centered around any frequency (modulo the higher order terms in the series).  Summing these terms and running a bootstrap argument to control the tail of the series leads directly to \eqref{apriori bds}, at least for $p$ large enough and $s$ small enough.

The weight $(4\varkappa^2+\xi^2)^{-1}$ appearing in \eqref{alpha2} decays too slowly to treat $p$ small and/or $s$ large by the method just outlined.  To overcome this obstacle, we employ a differencing technique from \cite{KVZ18}; see \eqref{beta defn}.  This was also essential in \cite{Kla23}.  As in that paper, we are also forced to evaluate the quartic term in the series exactly; the simple bound that follows from \eqref{AKNS I2} is not strong enough.

As noted earlier, the equicontinuity result of Theorem~\ref{T:main} is new in modulation spaces.  In prior works \cite{HKV,KV19} treating $H^s(\R)$ spaces, such equicontinuity claims have relied solely on the freedom of varying the spectral parameter in $\alpha(\varkappa;u)$.  This approach is incompatible with the manner in which the modulation norms are controlled through the perturbation determinant.

Given an $M^{s,2}_p$-equicontinuous set $Q$ of initial data, Proposition~\ref{P:equi} constructs a slowly growing weight so that $Q$ is a bounded set in the stronger modulation-type norm incorporating this extra weight.  We then show a priori bounds for this stronger norm in order to propagate equicontinuity.

The details of both the usual and additionally weighted a priori bounds in modulation spaces are presented in Section~\ref{a priori equi}.  It is there that we prove Theorem~\ref{T:main} by fully controlling the tail of the series (Lemmas~\ref{R6} and~\ref{L:beta4 small p}) and carrying out the bootstrap argument.

Section~\ref{well posed} is devoted to the proof of Theorem~\ref{T:Main}.  As we review in Lemma~\ref{L:embedding}, each of the modulation spaces covered by this theorem embeds into $H^\sigma(\R)$ for some $\sigma>-\frac12$.  In this way, we can guarantee the existence of solutions by appealing to the principal result of \cite{HKV}.  The principal question to be addressed is the $M^{s,2}_p\to C([-T,T];M^{s,2}_p)$ continuity of the resulting map.  This relies on two crucial ingredients: the equicontinuity proved in Theorem~\ref{T:main} and the $H^\sigma\to C([-T,T];H^\sigma)$ continuity proved in \cite{HKV}.  %Concretely, $M^{s,2}_p$ equicontinuity allows us to upgrate $H^{\sigma}$ convergence to $M^{s,2}_p$ convergence. % See Propostion~\ref{P:upgrade convg}

\subsection*{Acknowledgements} S.H. was supported by a Fulbright-Nehru Fellowship and a DST-INSPIRE grant (DST/INSPIRE/04/2022/001457); R.K. was supported by NSF grant DMS-2154022;  M.V. was supported by NSF grants DMS-2054194 and DMS-2348018; Y.Z. was supported by the National Key R\&D Program of China (No. 2022YFA1006700). S.H. and Y.Z. are grateful for the hospitality of UCLA and R.K. and M.V. are grateful for the hospitality of the Erwin Schr\"odinger International Institute for Mathematics and Physics during the program ``Nonlinear Waves and Relativity", where parts of this work were completed.   

\subsection{Notation} Our convention for the Fourier transform is
\begin{align}\label{FT}
\widehat f(\xi) = \tfrac{1}{\sqrt{2\pi}} \int_\R e^{-i\xi x} f(x)\,dx  \qquad \text{so} \qquad f(x) = \tfrac{1}{\sqrt{2\pi}} \int_\R e^{i\xi x} \widehat f(\xi)\,d\xi.
\end{align}
Correspondingly, $\widehat{fg} = (2\pi)^{-\frac12} \widehat f*\widehat g$.

Throughout this paper, we employ the standard notation $A \lesssim B$ to indicate that $A \leq CB$ for some constant $C > 0$.  If $A \lesssim B$ and $B \lesssim A$, we write $A \simeq B$.
Occasionally, we adjoin subscripts to this notation to indicate dependence of the constant $C$ on other parameters; for instance, we write $A \lesssim_{\delta} B$ when $A \leq
CB$ for some constant $C > 0$ depending on $\delta$.

We denote the Schwartz space by $\mathcal S(\R)$.  For $\sigma\in \R$, we write $H^\sigma(\R)$ for the Sobolev space defined as the completion of Schwartz space with respect to the~norm
$$
\|f\|_{H^\sigma} := \| \langle \xi\rangle^\sigma \widehat f\|_{L^2}.
$$
Here and throughout this paper, we use a slight modification of the usual bracket notation, namely,
$$\langle x\rangle := (4+|x|^2)^{\frac12}.$$
We make this modification so that $\log(\langle x\rangle)$ represents a natural logarithmically growing weight that does not vanish at the origin.

\section{Preliminaries}
\label{prelim}

\subsection{Modulation spaces and equicontinuity}

Essential to our proof of well-posedness for \eqref{mkdv} is the embedding of modulation spaces into the Sobolev spaces $H^\sigma(\R)$.  This allows us to exploit the definitive $H^s(\R)$ well-posedness result for \eqref{mkdv} from \cite{HKV}.

\begin{lemma}\label{L:embedding}
We have $M^{s,2}_p(\R)\hookrightarrow H^{\sigma}(\R)$ whenever
\begin{align}\label{262}
\left\{
\begin{array}{ll}
\sigma< s-\frac{1}{2}+\frac{1}{p}&\quad\text{if}\quad 2<p<\infty,\\
\sigma\leq s&\quad\text{if}\quad 1\leq p\leq 2.
\end{array}
\right.
\end{align}
In particular, for $1\leq p<\infty$ and $s>\max\{-\frac12,-\frac1p\}$, there exists $\sigma=\sigma(p,s)>-\frac12$ such that $M^{s,2}_p(\R)\hookrightarrow H^{\sigma}(\R)$.
\end{lemma}

\begin{proof}
We have 
\begin{align*}
\|u\|_{H^{\sigma}}
&\simeq \|\langle k\rangle^\sigma \|\widehat{u}\|_{L^2(I_k)}\|_{\ell^2_k}=\|\langle k\rangle^{\sigma-s}\langle k\rangle^{s} \|\widehat{u}\|_{L^2(I_k)}\|_{\ell^2_k}.
\end{align*}

For $p> 2$, we may use the H\"older inequality to bound
\begin{align*}
\|u\|_{H^{\sigma}}
&\lesssim\| \langle k\rangle^{\sigma-s}\|_{\ell_k^{2p/(p-2)}}\|\langle k\rangle^s \|\widehat{u}\|_{L^2(I_k)}\|_{\ell^p_k} \lesssim \|u\|_{M^{s,2}_p}
\end{align*}
provided 
$\frac{2p}{p-2}(\sigma-s)<-1$ or equivalently, $\sigma<s-\frac{1}{2}+\frac{1}{p}.$

For $\sigma\leq s$ and $1\leq p\leq 2$, we may simply bound
\begin{equation*}
\|u\|_{H^\sigma}\lesssim \|\langle k\rangle^s \|\widehat{u}\|_{L^2(I_k)}\|_{\ell^2_k}\lesssim \|\langle k\rangle^s \|\widehat{u}\|_{L^2(I_k)}\|_{\ell^p_k}=\|u\|_{M^{s,2}_p}. \qedhere
\end{equation*}
\end{proof}

As in the previous works \cite{BKV, HKNV, HKV, KLV:BO, KLV:CM, KV19} on sharp well-posedness for completely integrable systems in $H^s$ spaces, equicontinuity will play an important role in the proof of well-posedness.  In this paper, we need the proper notion of equicontinuity for the modulation spaces $M_p^{s,2}(\R)$.

\begin{definition}[Equicontinuity]\label{D:equi}  Fix $1\leq p<\infty$ and $s\geq 0$. We say that a bounded subset $Q$ of $M_p^{s,2}(\R)$ is \emph{equicontinuous} in $M_p^{s,2}(\R)$ if 
\begin{align*}
\lim_{K\to \infty}\, \sup_{f\in Q}\,\bigl\|\langle k\rangle^s \|\widehat f\|_{L^2(I_k)}\bigr\|_{\ell^p(|k|\ge K)}=0.
\end{align*} 
\end{definition}

\begin{remark}\label{R:equi}
Using the Plancherel theorem, we see that a bounded subset $Q$ of $M_p^{s,2}(\R)$ is equicontinuous in the sense above if and only if it satisfies the more traditional notion of equicontinuity:
 $$
\lim_{\delta\to 0} \,\sup_{f\in Q} \,\sup_{|y|\leq \delta}\,  \|f(x)-f(x+y)\|_{M^{s,2}_p} = 0.
$$
From either definition it is easy to verify that any set which is precompact in $M_p^{s,2}(\R)$ is equicontinuous. In particular, any Cauchy sequence is equicontinuous.
\end{remark}

\begin{proposition}[Characterization of equicontinuity]\label{P:equi}
Fix $1\leq p<\infty$ and $s\geq 0$ and let $Q\subset M_p^{s,2}(\R)$ be bounded. Then $Q$ is equicontinuous in $M_p^{s,2}(\R)$ if and only if there exists a sequence $\{c_k\}_{k\in\mathbb Z}$ with the following properties: \\[2mm]
\textsl{(i)} $1\leq c_k=c_{-k}\leq 1+\log(|k|+1)$, \\[2mm]
\textsl{(ii)} $c_{4|k|}\leq  c_{|k|}+1$, \\[2mm]
\textsl{(iii)} $c_{|k|}\leq c_{|k|+1}$, \\[2mm]
\textsl{(iv)} $\lim_{|k|\to\infty} c_{k}=+\infty$,  \\[2mm]
\textsl{(v)} $\sup_{f\in Q}\,\bigl\|c_k\langle k\rangle^s \|\widehat f\|_{L^2(I_k)}\bigr\|_{\ell_k^p}\le 2\sup_{f\in Q}\|f\|_{M_p^{s,2}}$.
\end{proposition}

\begin{proof}
By Definition~\ref{D:equi}, properties (iv) and (v) guarantee that the set $Q$ is equicontinuous in $M^{s,2}_p(\R)$.  Thus, it remains to prove that equicontinuity of $Q$ guarantees the existence of a sequence $c_k$ with the stated properties.

Let
$$
A:= \sup_{f\in Q}\|f\|_{M^{s,2}_p} <\infty.
$$
As $Q$ is equicontinuous in $M^{s,2}_p(\R)$, there exists an increasing sequence $\{k(m)\}_{m\geq 1}\subset \mathbb N$ such that 
\begin{align}\label{11:50}
k(m+1)\geq 4k(m) \quad\text{and}\quad \sup_{f\in Q}\, \sum_{|k|>k(m)} \langle k \rangle^{ps}\|\widehat{f}\|_{L^2(I_k)}^p\leq 2^{-m}A^p.
\end{align}

For $k\in\mathbb{Z}$ we define
$$
c_k=\Bigl[\#\{m\in\mathbb{N}: k(m)<|k|\}+1\Bigr]^{\frac1p}.
$$
In view of \eqref{11:50}, we have
\begin{align}\label{aksup}
\sup_{f\in Q}\, \bigl\|c_k\langle k \rangle^{s}\|\widehat{f}\|_{L^2(I_k)}\bigr\|_{\ell^p_k}^p \leq 2A^p,
\end{align}
which immediately implies property (v).  Properties (iii) and (iv) are also easily verified for this construction.  The sub-logarithmic growth in  property (i) and property (ii) follow from the observation that $k(m)\geq 4^{m-1}$.
\end{proof}

To upgrade $H^\sigma(\R)$-convergence of a sequence of Schwartz solutions to convergence in the modulation spaces $M^{s,2}_p(\R)$, we will prove and then exploit that the orbits of such solutions form an equicontinuous set in $M^{s,2}_p(\R)$.  This is formalized by the following:

\begin{proposition}\label{P:upgrade convg}
Fix $1\leq p<\infty$ and $s\geq 0$.  Let $\sigma=\sigma(s,p)>-\frac12$ be chosen in accordance with Lemma~\ref{L:embedding} so that $M^{s,2}_p(\R)\hookrightarrow H^{\sigma}(\R)$.  Suppose that the sequence $\{f_n\}_{n\geq 1}$ is bounded and equicontinuous in $M^{s,2}_p(\R)$ and that it converges in $H^\sigma(\R)$.  Then $\{f_n\}_{n\geq 1}$ converges in $M^{s,2}_p(\R)$.
\end{proposition}

\begin{proof}
As the modulation spaces $M^{s,2}_p(\R)$ are complete, it suffices to show that the sequence $\{f_n\}_{n\geq 1}$ is Cauchy in $M^{s,2}_p(\R)$.

The equicontinuity hypothesis guarantees that given $\varepsilon>0$, there exists $K_\varepsilon>0$ so that 
$$
\sup_{n\geq 1}\,\bigl\|\langle k\rangle^s \|\widehat f_n\|_{L^2(I_k)}\bigr\|_{\ell^p(|k|> K_\varepsilon)}<\tfrac\varepsilon4.
$$
In particular,
\begin{align}\label{high}
\bigl\| \langle k\rangle^s \|\widehat f_n -\widehat f_m\|_{L^2(I_k)} \bigr\|_{\ell^p(|k|> K_\varepsilon)}\leq 2\sup_{n\geq 1}\,\bigl\|\langle k\rangle^s \|\widehat f_n\|_{L^2(I_k)}\bigr\|_{\ell^p(|k|> K_\varepsilon)}\leq \tfrac\varepsilon2.
\end{align}

By \eqref{262} we have $\sigma\leq s$.  Thus, we may use the embedding $\ell^2\hookrightarrow \ell^p$ if $2\leq p<\infty$ and the H\"older inequality if $1\leq p<2$ to bound 
\begin{align*}
\bigl\| \langle k\rangle^s \|\widehat f_n &-\widehat f_m\|_{L^2(I_k)} \bigr\|_{\ell^p(|k|\leq K_\varepsilon)} \\
&\leq \langle K_\varepsilon\rangle^{s-\sigma} \bigl\| \langle k\rangle^\sigma \|\widehat f_n -\widehat f_m\|_{L^2(I_k)} \bigr\|_{\ell^p(|k|\leq K_\varepsilon)}\\
&\leq  \langle K_\varepsilon\rangle^{s-\sigma}  (2K_\varepsilon+1)^{\max\{0, \frac1p-\frac12\}}  \bigl\| \langle k\rangle^\sigma \|\widehat f_n -\widehat f_m\|_{L^2(I_k)} \bigr\|_{\ell^2(|k|\leq K_\varepsilon)}\\
&\lesssim \langle K_\varepsilon\rangle^{s-\sigma+ \max\{0, \frac1p-\frac12\}} \|f_n-f_m\|_{H^\sigma}.
\end{align*}
As $f_n$ converges in $H^\sigma(\R)$, we may find $N=N(\varepsilon, K_\varepsilon)$ so that $n, m\geq N$ guarantees
\begin{align}\label{low}
\bigl\| \langle k\rangle^s \|\widehat f_n -\widehat f_m\|_{L^2(I_k)} \bigr\|_{\ell^p(|k|\leq K_\varepsilon)}\leq \tfrac\varepsilon2.
\end{align}

Combining \eqref{high} and \eqref{low} with the triangle inequality, for $n, m\geq N$ we find
\begin{align*}
\|f_n-f_m\|_{M^{s,2}_p}&\leq \varepsilon,
\end{align*}
which demonstrates that $\{f_n\}_{n\geq 1}$ is Cauchy in $M^{s,2}_p(\R)$.
\end{proof}

\subsection{Conservation laws and the perturbation determinant}\label{SS:perturbation determinant}
Following \cite{KVZ18}, for $\varkappa>0$ we define
\begin{align}\label{alpha def}
\alpha(\varkappa;f)=\Re\sum_{j=1}^\infty\frac{(\mp1)^{j-1}}{j}\tr\left\{\left[(\varkappa-\partial)^{-\frac{1}{2}}f (\varkappa+\partial)^{-1}\bar{f}(\varkappa-\partial)^{-\frac{1}{2}}\right]^j\right\},
\end{align}
which formally agrees with the logarithm of the perturbation determinant.  The operators $(\varkappa\pm \partial)^{-\frac12}$ are defined via their Fourier symbols, where the branch of the square root is chosen to be positive at frequency zero.  The existence of the traces and the summability of the series in \eqref{alpha def} will be discussed below.

We write
$$
\alpha(\varkappa;f)=\sum_{j=1}^\infty\alpha_{2j}(\varkappa;f),
$$
where $\alpha_{2j}$ denotes the term in the series \eqref{alpha def} that is of order $2j$ in $f$, namely,
$$
\alpha_{2j}(\varkappa;f)=\Re\frac{(\mp1)^{j-1}}{j}\tr\left\{\left[(\varkappa-\partial)^{-\frac{1}{2}}f (\varkappa+\partial)^{-1}\bar{f}(\varkappa-\partial)^{-\frac{1}{2}}\right]^j\right\}.
$$

In \cite{KVZ18}, the authors proved that $\alpha$ is well-defined and conserved by Schwartz solutions to both \eqref{mkdv} and \eqref{nls}, provided $\varkappa$ is large.  Specifically, we have:

\begin{proposition}[Proposition 4.4 in \cite{KVZ18}]\label{P:alpha conserved}  There exists a small absolute constant $\eta>0$ so that the following holds:  If $u_0$ is a Schwartz function and $u(t)$ denotes the global solution to either \eqref{mkdv} or \eqref{nls} with initial data $u(0)=u_0$ and $\varkappa>0$ is large enough so that
\begin{align}\label{k large}
\int \log\bigl(4 + \tfrac{\xi^2}{\varkappa^2}\bigr)\frac{|\widehat u_0(\xi)|^2}{\sqrt{4\varkappa^2 + \xi^2}} \,d\xi < \eta,
\end{align}
then the series defining $\alpha(\varkappa;u(t))$ converges for all $t\in\R$ and 
\begin{equation*}
\frac{d}{dt} \alpha(\varkappa;u(t)) = 0.
\end{equation*}
\end{proposition}

The significance of condition \eqref{k large} is evident from the following:

\begin{lemma}[Lemma 4.1 in \cite{KVZ18}] \label{L:HS bdd}For $\varkappa>0$ and $f$ a Schwartz function, we have
\begin{equation}\label{AKNS I2}
\Bigl\|(\varkappa-\partial)^{-1/2} f (\varkappa+\partial)^{-1/2} \Bigr\|^2_{\mathfrak I_2} \simeq \int_\R \log\bigl(4 + \tfrac{\xi^2}{\varkappa^2}\bigr)\frac{|\widehat f(\xi)|^2}{\sqrt{4\varkappa^2 + \xi^2}}\,d\xi.
\end{equation}
Here $\mathfrak I_2$ denotes the Hilbert--Schmidt norm.
\end{lemma}

For $j\geq 2$, Lemma~\ref{L:HS bdd} allows us bound
\begin{align} \label{tailbound}
\bigl|\alpha_{2j}(\varkappa;f)\bigr|&\leq \Bigl\|(\varkappa-\partial)^{-1/2} f (\varkappa+\partial)^{-1/2} \Bigr\|^{2j}_{\mathfrak I_2}\notag\\
&\lesssim \biggl\{\int_\R \log\bigl(4 + \tfrac{\xi^2}{\varkappa^2}\bigr)\frac{|\widehat f(\xi)|^2}{\sqrt{4\varkappa^2 + \xi^2}}\,d\xi\biggr\}^j,
\end{align}
and so \eqref{k large} guarantees the convergence of the series \eqref{alpha def} with $f=u_0$ and $\eta$ sufficiently small.  The convergence of the series defining $\alpha(\varkappa; u(t))$ for all other $t\in \R$ is deduced in \cite{KVZ18} via a bootstrap argument.

\begin{lemma}\label{L:M controls HS}
For any $1\leq p<\infty$, $s\geq 0$, and $0<\delta<\min\{\frac12, \frac1p\}$, we have
\begin{align}\label{428}
\int \log\bigl(4 + \tfrac{\xi^2}{\varkappa^2}\bigr)\frac{|\widehat f(\xi)|^2}{\sqrt{4\varkappa^2 + \xi^2}} \,d\xi \lesssim_\delta \varkappa^{-2\delta} \|f\|_{M^{s,2}_p}^2,
\end{align}
uniformly for $\varkappa\geq \tfrac12$ and $f\in \mathcal S(\R)$.  
\end{lemma}

Given $u_0\in\mathcal S(\R)$, \eqref{428} shows that  \eqref{k large} holds whenever  $\varkappa\geq \varkappa_0(\eta, \delta,  \|u_0\|_{M^{s,2}_p})$.  It also shows that \eqref{k large} holds uniformly for all $\varkappa\geq\frac12$ if the modulation norm of the initial data is small enough.  In this paper, we will adopt the second approach, using the scaling symmetry to achieve the requisite smallness of the initial data; see Lemma~\ref{L:scaling} below.

\begin{proof}
For $0<\delta<\frac12$ and $\varkappa\geq \frac12$, we may bound
\begin{align*}
\int \log\bigl(4 + \tfrac{\xi^2}{\varkappa^2}\bigr)\frac{|\widehat f(\xi)|^2}{\sqrt{4\varkappa^2 + \xi^2}} \,d\xi
&\leq \frac1{\delta\varkappa^{2\delta}}\!\int \!\frac{|\widehat f(\xi)|^2}{(4\varkappa^2 + \xi^2)^{\frac12-\delta}} \,d\xi
\leq \frac1{\delta \varkappa^{2\delta}}\! \int\! \frac{|\widehat f(\xi)|^2}{(1 + \xi^2)^{\frac12-\delta}} \,d\xi.
\end{align*}
To continue, we decompose
\begin{eqnarray*}
\int  \frac{|\widehat f(\xi)|^2}{(1 + \xi^2)^{\frac12-\delta}} \,d\xi \simeq \sum_{m\in \Z}\langle m\rangle^{-1+2\delta}\|\widehat f\|_{L^2(I_m)}^2.
\end{eqnarray*}
For $1\le p\le2$ we use $\langle m\rangle^{-1+2\delta}\leq 1$ and the embedding $\ell^p\hookrightarrow \ell^2$ to bound
\begin{align*}
\sum_{m\in \Z}\langle m\rangle^{-1+2\delta}\|\widehat f\|_{L^2(I_m)}^2\leq \bigl\|\|\widehat f\|_{L^2(I_m)}\bigr\|_{\ell^p_m}^2 = \|f\|_{M^{0,2}_p}^2\leq\|f \|_{M^{s,2}_p}^2.
\end{align*}
For $2<  p<\infty$ we use that $0<\delta<\frac1p$ and the H\"older inequality to bound
\begin{align*}
 \sum_{m\in \Z}\langle m\rangle^{-1+2\delta}\|\widehat f\|_{L^2(I_m)}^2
&\lesssim \bigl\|\langle m\rangle^{-1+2\delta}\bigr\|_{\ell_m^{\frac{p}{p-2}}}\bigl\| \|\widehat f\|_{L^2(I_m)}\bigr\|_{\ell_m^{p}}^2\\
&\lesssim \|f\|_{M_p^{0,2}}^2\leq\|f\|_{M_p^{s,2}}^2.\qedhere
\end{align*}\end{proof}

As a convergent geometric series is as large as its leading term, we are naturally interested in the term $\alpha_2$.  This was computed explicitly in \cite{KVZ18}:

\begin{lemma}[Lemma 4.2 of \cite{KVZ18}]\label{L:alpha2}
For $\varkappa >0 $ and $f$ a Schwartz function, we have
\begin{align}\label{alpha2}
\alpha_2(\varkappa;f)=2\varkappa\int_\mathbb{R}\frac{|\widehat{f}(\xi)|^2}{4\varkappa^2+\xi^2}\ d\xi.
\end{align}
\end{lemma}

To obtain a priori bounds in modulation spaces for the whole range of parameters  claimed in Theorem~\ref{T:main}, we need two further ingredients.  The first is an explicit computation of the quartic term in the definition of $\alpha$:

\begin{lemma}\label{L:quartic}
For $\varkappa >0 $ and $f$ a Schwartz function, we have
$$
\alpha_4(\varkappa;f)=\mp\Re\!\iiint \tfrac{[2\varkappa (\xi_1\xi_2+\xi_1\xi_4+\xi_2\xi_4)-8\varkappa^3]\overline{\widehat{f}(\xi_1)\widehat{f}(\xi_3)}\widehat{f}(\xi_2)\widehat{f}(\xi_4)}{2\pi(4\varkappa^2+\xi_1^2)(4\varkappa^2+\xi_2^2)(4\varkappa^2+\xi_4^2)}\, d\xi_1\,d\xi_2\, d\xi_3,
$$
where $\xi_4=\xi_1-\xi_2+\xi_3$.
\end{lemma}

\begin{proof}
This formula appeared in the proof of a priori bounds in $M^{0,2}_p(\R)$ for \eqref{nls} in \cite{Kla23}. It was derived by equating $\alpha_4$ with the real part of the second term in the series expansion of the logarithm of the inverse transmission coefficient as computed in \cite{KT18}. For the sake of completeness, we provide a proof via direct calculation.

In Fourier variables, we have
\begin{align*}
\alpha_4(\varkappa;f)&= \mp\Re\int\cdots\int\tfrac{\widehat{f}(\theta_1-\theta_4) \widehat{\overline{f}}(\theta_4-\theta_3)\widehat{f}(\theta_3-\theta_2)\widehat{\overline{f}}(\theta_2-\theta_1)}{8\pi^2(\varkappa-i\theta_1)(\varkappa+i\theta_4)(\varkappa-i\theta_3)(\varkappa+i\theta_2)}\,d\theta_1\, d\theta_2\, d\theta_3\, d\theta_4 \\
&= \mp\Re\int\cdots\int\tfrac{\widehat{f}(\theta_1-\theta_4) \overline{\widehat f(\theta_3-\theta_4)}\widehat{f}(\theta_3-\theta_2)\overline{\widehat f(\theta_1-\theta_2)}}{8\pi^2(\varkappa-i\theta_1)(\varkappa+i\theta_4)(\varkappa-i\theta_3)(\varkappa+i\theta_2)}\,d\theta_1\, d\theta_2\, d\theta_3\, d\theta_4.
\end{align*}
Changing variables to $\xi_1=\theta_1-\theta_2$, $\xi_2=\theta_3-\theta_2$, $\xi_3=\theta_3-\theta_4$ so that $\theta_2=\theta_1-\xi_1$, $\theta_3=\theta_1-\xi_1+\xi_2$, and $\theta_4=\theta_1-\xi_1+\xi_2-\xi_3$, we obtain
\begin{align}\label{alpha4f1}
\alpha_4(\varkappa;f)&= \mp\Re \int\cdots\int\tfrac{\widehat{f}(\xi_1-\xi_2+\xi_3)\overline{\widehat f(\xi_3)}\widehat{f}(\xi_2)
    \overline{\widehat f(\xi_1)}}{8\pi^2(\theta_1-a_1)(\theta_1-a_2)(\theta_1-a_3)(\theta_1-a_4)}\, d\theta_1\, d\xi_1\, d\xi_2\, d\xi_3
\end{align}
where $a_1=-i\varkappa$, $a_2=i\varkappa+\xi_1$, $a_3=-i\varkappa+\xi_1-\xi_2$, and $a_4=i\varkappa+\xi_1-\xi_2+\xi_3$.  As $a_1$ and $a_3$ belong to  the lower half of the complex plane when $\varkappa>0$, the Cauchy integral formula gives 
\begin{align*}
\int \tfrac{ d\theta_1}{(\theta_1-a_1)(\theta_1-a_2)(\theta_1-a_3)(\theta_1-a_4)}= - 2\pi i\, (A_1+A_3)
\end{align*}
where 
$$
A_1=\tfrac{1}{(a_1-a_2)(a_1-a_3)(a_1-a_4)} \quad\text{and}\quad  A_3=\tfrac{1}{(a_3-a_1)(a_3-a_2)(a_3-a_4)}
$$
denote the residues of the integrand at $a_1$ and $a_3$, respectively. The claim now follows from a straightforward computation combining this with \eqref{alpha4f1} and symmetrizing this expression with respect to the changes of variables $\xi_1\leftrightarrow \xi_3$, $\xi_2\leftrightarrow \xi_4$, and $(\xi_1,\xi_3)\leftrightarrow(\xi_2,\xi_4)$. 
\end{proof}

To derive the claimed a priori bounds in modulation spaces for the whole range of parameters in Theorem~\ref{T:main}, we also need a further refinement of $\alpha$.  Specifically, we define
\begin{equation}\label{beta defn}
\beta(\varkappa;f):=\alpha(\varkappa;f)-\tfrac 12 \alpha(2\varkappa;f)=:\sum_{j=1}^\infty \beta_{2j}(\varkappa;f).
\end{equation}
In view of Proposition~\ref{P:alpha conserved}, $\beta$ is also conserved by Schwartz solutions to \eqref{mkdv} and \eqref{nls}, provided condition \eqref{k large} is satisfied.
A direct computation using \eqref{alpha2} yields
\begin{align*}
\beta_{2}(\varkappa;f) = 24\varkappa^3\int_\mathbb{R}\frac{|\widehat{f}(\xi)|^2}{(4\varkappa^2+\xi^2)(16\varkappa^2+\xi^2)}\, d\xi.
\end{align*}
The key virtue of this differencing technique (introduced already in \cite{KVZ18} and used also in \cite{Kla23}), is the improved decay as $\xi\to \infty$ relative to the earlier \eqref{alpha2}.

To relate this quantity to the modulation norms, we employ \emph{Galilean boosts}. Given a solution $u(t,x)$ to \eqref{mkdv} and a wave number $k\in \R$, the function
\begin{align}\label{Galilei mkdv}
u^k(t,x):=e^{-ikx+2ik^3t}u(t,x-3k^2t)
\end{align}
solves the following mixed equation:
\begin{align}\tag{mKdV--NLS}\label{mkdv-nls}
\partial_t u^k+ \partial_x^3 u^k= \pm 6|u^k|^2\partial_x u^k- 3ik\partial_x^2 u^k\pm 6ik|u^k|^2u^k.
\end{align}
In general, the Galilei transformation of any member of the NLS--mKdV hierarchy involves all lower flows in the hierarchy.  For \eqref{nls}, this comprises just translation and phase rotation, so the Galilei symmetry takes a particularly simple form:
\begin{align}\label{Galilei nls}
u(t,x)\mapsto u^k(t,x):=e^{-ikx-ik^2t}u(t,x+2kt)
\end{align}
preserves the class of solutions to \eqref{nls}.

Because $\alpha$ is conserved by Schwartz solutions to both \eqref{mkdv} and \eqref{nls} (cf. Proposition~\ref{P:alpha conserved}), the quantities $\alpha(\varkappa;u^k)$ and $\beta(\varkappa; u^k)$ are conserved by Schwartz solutions to \eqref{mkdv-nls}. Observe also that in the setting of both \eqref{mkdv} and \eqref{nls}, we have
\begin{eqnarray*}
|\widehat{u^k}(t,\xi)| =|\widehat{u}(t,\xi+k)|
\end{eqnarray*}
and so
\begin{align}\label{beta2}
\beta_{2}\bigl(\varkappa; u^k(t)\bigr) = 24\varkappa^3\int_\mathbb{R}\frac{|\widehat u(t,\xi)|^2}{[4\varkappa^2+(\xi-k)^2][16\varkappa^2+(\xi-k)^2]}\, d\xi.
\end{align}

The following result relates the quadratic part of $\beta$ to the modulation norms.  To treat a priori bounds and equicontinuity on an equal footing in our arguments, we prove a slightly stronger statement:

\begin{lemma}\label{L:beta mod norm}
Fix $1\le p<\infty$ and $0\leq s<2-\frac1p$.  Let $\{c_k\}_{k\in\mathbb Z}$ be a sequence such that either $c_k\equiv 1$ or $c_k$ satisfies properties (i)--(iv) from Proposition~\ref{P:equi}.  Let $u_0$ be a Schwartz function such that \eqref{k large} holds for all $\varkappa\geq \frac12$ and let $u(t)$ denote the global solution to either \eqref{mkdv} or \eqref{nls} with initial data $u(0)=u_0$. Then uniformly for $t\in \R$ we have 
\begin{align}\label{equiv norms}
\bigl\|c_k\langle k\rangle^s \|\widehat u(t)\|_{L^2(I_k)} \bigr\|_{\ell^p_k} \simeq \bigl\|c_k\langle k\rangle^s \sqrt{\beta_{2}\bigl(\tfrac12; u^k(t)\bigr)} \bigr\|_{\ell^p_k}.
\end{align}
\end{lemma}

\begin{remark}
If $c_k\equiv 1$, then $\text{LHS}\eqref{equiv norms}=\|u(t)\|_{M^{s,2}_p}$.  The scenario described in Lemma~\ref{L:beta mod norm} when $\{c_k\}_{k\geq 1}$ satisfies properties (i)--(iv) in Proposition~\ref{P:equi} will allow us to prove the equicontinuity statement in Theorem~\ref{T:main}.
\end{remark}

\begin{proof}
From \eqref{beta2} it is immediate that
$$
\beta_{2}\bigl(\tfrac12; u^k(t)\bigr)\gtrsim  \|\widehat u(t)\|_{L^2(I_k)}^2,
$$
and so  $\text{LHS}\eqref{equiv norms}\lesssim \text{RHS}\eqref{equiv norms}$.  On the other hand, using the subadditivity of fractional powers we may bound
\begin{align*}
\sqrt{\beta_{2}\bigl(\tfrac12; u^k(t)\bigr)} \simeq \Bigl\{\sum_{m\in \Z} \langle m-k\rangle^{-4}  \|\widehat u(t)\|_{L^2(I_m)}^2 \Bigr\}^{\frac12} \lesssim \sum_{m\in \Z} \langle m-k\rangle^{-2}  \|\widehat u(t)\|_{L^2(I_m)}.
\end{align*}
Thus, recalling the properties of the sequence $c_k$ and distinguishing the cases $2|m|\geq |k|$ and $2|m|<|k|$, we may bound
\begin{align}\label{9:28}
\bigl\|c_k\langle k\rangle^s \sqrt{\beta_{2}\bigl(\tfrac12; u^k(t)\bigr)} \bigr\|_{\ell^p_k}
&\lesssim \Bigl\|\sum_{|m|\geq \frac12|k|} c_k\langle k\rangle^s \langle m-k\rangle^{-2}\|\widehat u(t)\|_{L^2(I_m)} \Bigr\|_{\ell^p_k}\notag\\
&\quad+\Bigl\|\sum_{|m|< \frac12|k|} c_k\langle k\rangle^s \langle m-k\rangle^{-2}\|\widehat u(t)\|_{L^2(I_m)} \Bigr\|_{\ell^p_k}\notag\\
&\lesssim \Bigl\|\sum_{|m|\geq \frac12|k|}  \langle m-k\rangle^{-2} c_m\langle m\rangle^s\|\widehat u(t)\|_{L^2(I_m)} \Bigr\|_{\ell^p_k}\\
&\quad+\Bigl\|\sum_{|m|< \frac12|k|} c_k\langle k\rangle^{s-2} \|\widehat u(t)\|_{L^2(I_m)} \Bigr\|_{\ell^p_k}.\notag
\end{align} 
By Young's convolution inequality, we may bound the first term above as follows:
$$
\Bigl\|\sum_{m\in \Z}  \langle m-k\rangle^{-2} c_m\langle m\rangle^s\|\widehat u(t)\|_{L^2(I_m)} \Bigr\|_{\ell^p_k} \lesssim \bigl\|c_m\langle m\rangle^s\|\widehat u(t)\|_{L^2(I_m)} \bigr\|_{\ell^p_m},
$$
which is acceptable.

To bound the second term, we use H\"older and the properties of the sequence $c_k$:
\begin{align*}
\Bigl\|\sum_{|m|< \frac12|k|} &c_k\langle k\rangle^{s-2} \|\widehat u(t)\|_{L^2(I_m)} \Bigr\|_{\ell^p_k}\\
&\lesssim\Bigl\| c_k\langle k\rangle^{s-2}\bigl\|c_m\langle m\rangle^s\|\widehat u(t)\|_{L^2(I_m)} \bigr\|_{\ell^p_m}  \|\langle m\rangle^{-s}\|_{\ell_m^{p'}(|m|< \frac12|k|)} \Bigr\|_{\ell^p_k}\\
&\lesssim \bigl\|c_m\langle m\rangle^s\|\widehat u(t)\|_{L^2(I_m)} \bigr\|_{\ell^p_m} \Bigl\| c_k\langle k\rangle^{s-2}\max\bigl\{1,\langle k\rangle^{-s+\frac1{p'}}\log^{\frac1{p'}}(\langle k\rangle)\bigr\} \Bigr\|_{\ell^p_k}\\
&\lesssim \bigl\|c_m\langle m\rangle^s\|\widehat u(t)\|_{L^2(I_m)}\bigr\|_{\ell^p_m}.
\end{align*}
The presence of the term $\log^{\frac1{p'}}(\langle k\rangle)$ in the second inequality above is only necessary when $sp'=1$. Note that the sub-logarithmic growth of the sequence $c_k$ and our hypothesis $0\leq s<2-\frac1p$ guarantee that $c_k\langle k\rangle^{s-2}\in \ell^p_k$. Similarly, $c_k\langle k\rangle^{-2+\frac1{p'}}\log^{\frac1{p'}}(\langle k\rangle)\in \ell^p_k$ follows from the sub-logarithmic growth of the sequence $c_k$ and the fact that $-2+\frac{1}{p'}<-\frac1p$ for all $1\leq p<\infty$.
\end{proof}

\begin{remark}
It is essential in the proof above that we are working with the quadratic part of $\beta$, rather than that of $\alpha$.  Indeed, replacing $\beta_2$ with $\alpha_2$, we would have to replace the factor $\langle m-k\rangle^{-2}$ in \eqref{9:28} with $\langle m-k\rangle^{-1}$.  Consequently, the application of the Young convolution inequality would fail.  Similarly,  in the last step of the proof we would have to contend with the fact that the sequence $c_k\langle k\rangle^{-1+\frac1{p'}}\log^{\frac1{p'}}(\langle k\rangle)$ does not belong to $\ell^p$.
\end{remark}

\subsection{The scaling symmetry}\label{SS:tss}
The mKdV equation enjoys the following scaling symmetry: the mapping
$$
u(t,x)\mapsto u_\lambda(t,x):=\lambda^{-1} u(\lambda^{-3}t,\lambda^{-1}x), \qquad \lambda>0,
$$
preserves the class of solutions.  Similarly, \eqref{nls} enjoys the scaling symmetry
$$
u(t,x)\mapsto u_\lambda(t,x):=\lambda^{-1} u(\lambda^{-2}t,\lambda^{-1}x), \qquad \lambda>0.
$$

The next lemma shows how scaling interacts with the modulation norms:

\begin{lemma}\label{L:scaling}
Fix $1\leq p<\infty$ and $s\geq 0$.  For  $\lambda>0$ and $f\in M^{s,2}_p(\R)$, let $f_\lambda(x) = \lambda^{-1} f(\lambda^{-1}x)$.  Then
\begin{align}\label{scaling1}
\|f_\lambda\|_{M^{s,2}_p}\lesssim\langle\lambda\rangle^{-\min\{\frac12,\frac1p\}}\langle\lambda^{-1}\rangle^{s+\max\{\frac12,\frac1p\}}\|f\|_{M^{s,2}_p}.
\end{align}
%and 
%\begin{align}\label{scaling2}
%\|u(\lambda^{-3}t)\|_{M^{s,2}_p}\lesssim \lambda^{s+\max\{\frac12,\frac1p\}}\|u_\lambda(t)\|_{M^{s,2}_p}.
%\end{align}
\end{lemma}

\begin{proof}
As
$$
\widehat{f_\lambda}(\xi)=\widehat{f}(\lambda \xi),
$$
performing a change of variables we get
\begin{align*}
\|f_\lambda\|_{M^{s,2}_p}
&=\bigl\|\langle k\rangle^s\|\widehat{f}(\lambda\xi)\|_{L^2_\xi(I_k)}\bigr\|_{\ell^p_k}
=\lambda^{-\frac 12}\bigl\|\langle  k\rangle^s\|\widehat{f}\|_{L^2(\lambda\cdot I_k)}\bigr\|_{\ell^p_k},
\end{align*}
where $\lambda\cdot I_k=[\lambda k-\lambda/2,\lambda k+\lambda/2)$.  The intersections between  $\lambda\cdot I_k$ and the unit intervals $I_m$ are reflected by the sets
$$
\mathcal{I}_{\lambda,k} :=\{m\in\mathbb{Z}: I_m\cap \lambda \cdot I_k\neq\emptyset\}
	\qtq{and} \mathcal{J}_{\lambda,m} :=\{k\in\mathbb{Z}: I_m\cap \lambda \cdot I_k\neq\emptyset\},
$$
which satisfy the following cardinality bounds:
$$
|\mathcal{I}_{\lambda,k}|\lesssim\langle\lambda\rangle \qtq{and} |\mathcal{J}_{\lambda,m}|\lesssim \langle\lambda^{-1}\rangle.
$$
From this point, we will treat the cases $\lambda\geq 1$ and $\lambda\leq 1$ separately.

Suppose $\lambda\geq 1$.   For $1\leq p\leq 2$, we use the subadditivity of fractional powers together with the easy estimate $\langle k\rangle\lesssim \langle m\rangle$ whenever $m\in \mathcal{I}_{\lambda,k}$ to bound
\begin{align*}
  \|f_\lambda\|_{M^{2,p}_s}
= \lambda^{-\frac12}\bigl\|\langle  k\rangle^s\|\widehat{f}\|_{L^2(\lambda\cdot I_k)}\bigr\|_{\ell^p_k}
&\lesssim\lambda^{-\frac12}\left\|\sqrt{\sum_{m\in \mathcal{I}_{\lambda,k}}\bigl\langle k\bigr\rangle^{2s} \|\widehat{f}\|^2 _{L^2(I_m)}}\, \right\|_{\ell^p_k}\\
&\lesssim\lambda^{-\frac12}\biggl\{\,\sum_{k\in\mathbb{Z}}\ \sum_{m\in \mathcal{I}_{\lambda,k}}\langle m\rangle^{ps} \|\widehat{f}\|^{p}_{L^2(I_m)}\biggr\}^{\frac1p}\\
&\lesssim \lambda^{-\frac12} \|f\|_{M^{s,2}_p},
\end{align*}
which yields \eqref{scaling1} in this case.  The last inequality relies on $|\mathcal{J}_{\lambda,m}| \lesssim 1$ when $\lambda\geq 1$.

For $2<p<\infty$, we use instead the H\"older inequality to estimate 
\begin{align*}
\|f_\lambda\|_{M^{s,2}_p}
&\lesssim \lambda^{-\frac12}\biggl\|\sum_{m\in \mathcal{I}_{\lambda,k}}\langle m\rangle^{2s} \|\widehat{f}\|^2 _{L^2(I_m)}\biggr\|^{1/2}_{\ell^{p/2}_k}\\
&\lesssim\lambda^{-\frac12}\biggl\|\,|\mathcal{I}_{\lambda,k}|^{1-\frac2p}\left\|\langle m\rangle^{2s} \|\widehat{f}\|^2_{L^2(I_m)}\right\|_{\ell^{p/2}_m(\mathcal{I}_{\lambda,k})}\biggr\|^{1/2}_{\ell^{p/2}_k}\\
&\lesssim\lambda^{-\frac12}\lambda^{\frac12-\frac1p}\biggl\|\left\|\langle m\rangle^{2s} \|\widehat{f}\|^2 _{L^2(I_m)}\right\|_{\ell^{p/2}_m(\mathcal{I}_{\lambda,k})}\biggr\|^{1/2}_{\ell^{p/2}_k}\\
&\lesssim  \lambda^{-\frac1p} \|f\|_{M^{s,2}_p},
\end{align*}
which implies \eqref{scaling1} in this case.

Suppose now that $\lambda\leq1$.  The central observation is that
\begin{align*}
\lambda^{-\frac12}\bigl\|\langle  k\rangle^s\|\widehat{f}\|_{L^2(\lambda\cdot I_k)}\bigr\|_{\ell^p_k(\mathcal J_{\lambda,m})}
&\lesssim \lambda^{-\frac12 - s} \langle  m\rangle^s \, | \mathcal J_{\lambda,m}|^{\max\{0,\frac1p-\frac12\}} \, \|\widehat{f}\|_{L^2(I_m)} \\
&\lesssim \lambda^{-\frac12 - s-\max\{0,\frac1p-\frac12\}} \,  \langle  m\rangle^s\|\widehat{f}\|_{L^2(I_m)}.
\end{align*}
For $2\leq p<\infty$, this follows from the embedding $\ell^2\hookrightarrow\ell^p$.  For $1\leq p<2$, one uses H\"older's inequality.  To finish the proof of \eqref{scaling1}, it remains to take the $\ell_m^p$ norm of both sides.
\end{proof}

Combining Lemmas~\ref{L:M controls HS} and \ref{L:scaling}, we see that given any Schwartz solution $u(t)$ to \eqref{mkdv} or \eqref{nls}, we may employ the scaling symmetries of these equations to guarantee that \eqref{k large} holds for $u_\lambda(0)$ for all $\varkappa\geq \frac12$.   Indeed, it is apparent from \eqref{scaling1} that as $\lambda\to \infty$, the norm of the rescaled initial data converges to zero.  In particular, Lemma~\ref{L:beta mod norm} applies to the rescaled solution $u_\lambda(t)$ when $\lambda$ is large.

Further, it is easy to see that scaling preserves the equicontinuity property.  Specifically, if $Q\subset M^{s,2}_p(\R)$ is a bounded and equicontinuous set and $\lambda>0$ is fixed, then the set $Q_\lambda:=\{ f_\lambda (x)= \lambda^{-1} f( \lambda^{-1}x):\, f\in Q\}$ is also bounded and equicontinuous in $M^{s,2}_p(\R)$.

\section{A priori bounds and equicontinuity}
\label{a priori equi}

The goal of this section is to prove Theorem~\ref{T:main}.  Fix $u_0\in \mathcal S(\mathbb\R)$ and let $u(t)$ denote the global solution to either \eqref{mkdv} or \eqref{nls} with initial data $u(0)=u_0$.  As we will explain more fully at the end of this section, by employing the scaling symmetries of \eqref{mkdv} and \eqref{nls} together with Lemma~\ref{L:M controls HS}, it suffices to prove Theorem~\ref{T:main} for small initial data.  Specifically, we will first establish the claim under the assumption 
\begin{align}\label{small norm data}
\|u_0\|_{M^{s,2}_p}\leq \varepsilon
\end{align}
for some small parameter $\varepsilon>0$ to be chosen later.  By the time-continuity of the \eqref{mkdv} and \eqref{nls} flows in Schwartz space, there exists an open interval $I$ containing $0$ such that
\begin{align}\label{small norm}
\sup_{t\in I} \,\|u(t)\|_{M^{s,2}_p}\leq 2C_0\varepsilon,
\end{align}
where $C_0\geq 1$ is a fixed universal constant that will subsume all implicit constants in the computations that follow.

In view of Lemma~\ref{L:M controls HS}, taking $\varepsilon=\varepsilon(C_0,\eta)$ small we may ensure that
\begin{align}\label{smalltail}
\sup_{t\in I }\, \int \log\bigl(4 + \tfrac{\xi^2}{\varkappa^2}\bigr)\frac{|\widehat u(t,\xi)|^2}{\sqrt{4\varkappa^2 + \xi^2}} \,d\xi <\eta
\end{align}
for all $\varkappa\geq \frac12$, where $\eta$ is as dictated by Proposition~\ref{P:alpha conserved}.

To treat a priori bounds and equicontinuity concomitantly, we let $\{c_k\}_{k\in\mathbb Z}$ be a sequence such that either $c_k\equiv 1$ or $c_k$ satisfies properties (i)---(iv) from Proposition~\ref{P:equi}.  To keep formulas within margins, we adopt the notation
$$
\beta_{\geq 2\ell}(\varkappa;f) := \sum_{j=\ell}^\infty \beta_{2j}(\varkappa;f).
$$

In view of \eqref{tailbound} and \eqref{smalltail}, for all $k\in\Z$ we have 
\begin{align}\label{beta tail}
\bigl|\beta_{\geq 2\ell}\bigl(\tfrac12;u^k(t)\bigr)\bigr|&\lesssim \biggl\{\int_\R \frac{\log(\langle \xi-k\rangle)}{\langle \xi-k\rangle} \, |\widehat u(t,\xi)|^2\,d\xi\biggr\}^\ell,
\end{align}
where $u^k$ denotes the Galilei boosted solution as defined in \eqref{Galilei mkdv} and \eqref{Galilei nls} for \eqref{mkdv} and \eqref{nls}, respectively.

Using Lemma~\ref{L:beta mod norm} and the conservation of $\beta$, we may bound
\begin{align}
\bigl\|c_k\langle k\rangle^s\|\widehat{u}(t,\xi)\|_{L^2_\xi(I_k)}\bigr\|_{\ell_k^p}
&\lesssim \Bigl\|c_k\langle k\rangle^{s}\sqrt{\beta\bigl(\tfrac12; u^k(t)\bigr)}\Bigr\|_{\ell_k^{p}}
+\Bigl\|c_k\langle k\rangle^{s}\sqrt{\beta_{\geq 4}\bigl(\tfrac12; u^k(t)\bigr)}\Bigr\|_{\ell_k^{p}}\notag\\
& \lesssim \bigl\|c_k\langle k\rangle^s\|\widehat{u}(0,\xi)\|_{L^2_\xi(I_k)}\bigr\|_{\ell_k^p}
+\Bigl\|c_k\langle k\rangle^{s}\sqrt{\beta_{\geq 4}\bigl(\tfrac12; u^k(0)\bigr)}\Bigr\|_{\ell_k^{p}}\label{3:33}\\
&\quad +\Bigl\|c_k\langle k\rangle^{s}\sqrt{\beta_{\geq 4}\bigl(\tfrac12; u^k(t)\bigr)}\Bigr\|_{\ell_k^{p}}.\notag
\end{align}

In the next two lemmas, we bound the contribution of the quartic and higher order terms in the expression above.  To cover the whole range of the parameters in Theorem~\ref{T:main}, we have to distinguish the quartic terms from the sextic and higher order terms  in our analysis.  We begin with $\beta_{\geq 6}$.

\begin{lemma}\label{R6}
Fix $1\leq p<\infty$ and $0\leq s<\frac32-\frac1p$.  Then 
\begin{align*}
\Bigl\|c_k\langle k\rangle^{s}\sqrt{\beta_{\geq 6}\bigl(\tfrac12; u^k(t)\bigr)}\Bigr\|_{\ell_k^{p}}\lesssim \bigl\|c_k\langle k\rangle^{s}\|\widehat{u}(t,\xi)\|_{L_\xi^2(I_k)}\bigr\|_{\ell_k^p}^3
\end{align*}
uniformly for $t\in I$.
\end{lemma}

\begin{proof}
Using \eqref{beta tail}, recalling the properties of the sequence $c_k$, and separating the cases $2|m|\geq |k|$ and $2|m|<|k|$, we may bound
\begin{align}\label{1:41}
\Bigl\|c_k\langle k\rangle^{s}\sqrt{\beta_{\geq 6}\bigl(\tfrac12; u^k(t)\bigr)}\Bigr\|_{\ell_k^{p}}
&\lesssim \biggl\|\sum_{m\in \Z} c_k^{\frac23}\langle k\rangle^{\frac{2s}3} \tfrac{\log(\langle m-k\rangle)}{\langle m-k\rangle} \|\widehat u(t)\|_{L^2(I_m)}^2\biggr\|_{\ell_k^{\frac{3p}2}}^{\frac32}\notag\\
&\lesssim \biggl\|\sum_{|m|\geq \frac12|k|} \tfrac{\log(\langle m-k\rangle)}{\langle m-k\rangle} c_m^{2}\langle m\rangle^{2s} \|\widehat u(t)\|_{L^2(I_m)}^2\biggr\|_{\ell_k^{\frac{3p}2}}^{\frac32}\\
&\quad+ \biggl\|\sum_{|m|< \frac12|k|}c_k^{\frac23}\langle k\rangle^{\frac{2s}3 -1} \log(\langle k\rangle) \|\widehat u(t)\|_{L^2(I_m)}^2\biggr\|_{\ell_k^{\frac{3p}2}}^{\frac32}.\notag
\end{align}

We first consider the case $1\leq p\leq 2$.  Using the fact that $\ell^p\hookrightarrow\ell^2$, we may bound 
\begin{align*}
\biggl\|\sum_{|m|\geq \frac12|k|}& \tfrac{\log(\langle m-k\rangle)}{\langle m-k\rangle} c_m^{2}\langle m\rangle^{2s} \|\widehat u(t)\|_{L^2(I_m)}^2\biggr\|_{\ell_k^{\frac{3p}2}}\\
&\lesssim \sum_{m\in \Z} c_m^{2}\langle m\rangle^{2s} \|\widehat u(t)\|_{L^2(I_m)}^2 \bigl\| \tfrac{\log(\langle k\rangle)}{\langle k\rangle}\bigr\|_{\ell_k^{\frac{3p}2}}
\lesssim \bigl\|  c_m\langle m\rangle^{s} \|\widehat u(t)\|_{L^2(I_m)} \bigr\|_{\ell_m^p}^2
\end{align*}
and
\begin{align*}
\biggl\|\sum_{|m|< \frac12|k|}  c_k^{\frac23}\langle k\rangle^{\frac{2s}3 -1} \log(\langle k\rangle) &\|\widehat u(t)\|_{L^2(I_m)}^2\biggr\|_{\ell_k^{\frac{3p}2}}\\
&\lesssim  \sum_{m\in \Z}\|\widehat u(t)\|_{L^2(I_m)}^2 \bigl\|c_k^{\frac23}\langle k\rangle^{\frac{2s}3 -1} \log(\langle k\rangle)\bigr\|_{\ell_k^{\frac{3p}2}}\\
&\lesssim \sum_{m\in \Z} c_m^{2}\langle m\rangle^{2s} \|\widehat u(t)\|_{L^2(I_m)}^2\\
&\lesssim  \bigl\| c_m\langle m\rangle^{s} \|\widehat u(t)\|_{L^2(I_m)} \bigr\|_{\ell_m^p}^2.
\end{align*}
To obtain the second inequality above, we used that $c_m^{2}\langle m\rangle^{2s} \geq 1$ for all $m\in \Z$ together with the fact $c_k^{\frac23}\langle k\rangle^{\frac{2s}3 -1} \log(\langle k\rangle)\in \ell_k^{\frac{3p}2}$ whenever  $s<\frac32-\frac1p$.

We now turn to the case $2<p<\infty$.  Using Young's convolution inequality, we may bound
\begin{align*}
\biggl\|\sum_{|m|\geq \frac12|k|}& \tfrac{\log(\langle m-k\rangle)}{\langle m-k\rangle} c_m^{2}\langle m\rangle^{2s} \|\widehat u(t)\|_{L^2(I_m)}^2\biggr\|_{\ell_k^{\frac{3p}2}}\\
&\lesssim \bigl\|  c_m^{2}\langle m\rangle^{2s} \|\widehat u(t)\|_{L^2(I_m)}^2 \bigr\|_{\ell_m^{\frac p2}} \bigl\| \tfrac{\log(\langle m\rangle)}{\langle m\rangle}\bigr\|_{\ell_m^{\frac{3p}{3p-4}}}
\lesssim \bigl\|  c_m\langle m\rangle^{s} \|\widehat u(t)\|_{L^2(I_m)} \bigr\|_{\ell_m^p}^2,
\end{align*}
which provides an acceptable bound for the first term on the right-hand side of \eqref{1:41}.  Considering the second term, we use the H\"older inequality to get
\begin{align*}
 &\biggl\|\sum_{|m|< \frac12|k|} c_k^{\frac23}\langle k\rangle^{\frac{2s}3 -1} \log(\langle k\rangle) \|\widehat u(t)\|_{L^2(I_m)}^2\biggr\|_{\ell_k^{\frac{3p}2}}\\
&\lesssim  \biggl\| c_k^{\frac23}\langle k\rangle^{\frac{2s}3 -1}\log(\langle k\rangle)  \bigl\|  c_m^{2}\langle m\rangle^{2s} \|\widehat u(t)\|_{L^2(I_m)}^2 \bigr\|_{\ell_m^{\frac p2}}\|\langle m\rangle^{-2s}\|_{\ell_m^{\frac{p}{p-2}}(|m|< \frac12|k|)}\biggr\|_{\ell_k^{\frac{3p}2}}\\
&\lesssim  \bigl\|  c_m\langle m\rangle^{s} \|\widehat u(t)\|_{L^2(I_m)} \bigr\|_{\ell_m^p}^2  \Bigl\|c_k^{\frac23}\langle k\rangle^{\frac{2s}3 -1}\!\log(\langle k\rangle) \max\bigl\{1,\! \langle k\rangle^{-2s+\frac{p-2}p}\!\log^{\frac{p-2}{p}}(\langle k\rangle)\bigr\} \Bigr\|_{\ell_k^{\frac{3p}2}}\\
&\lesssim  \bigl\|  c_m\langle m\rangle^{s} \|\widehat u(t)\|_{L^2(I_m)} \bigr\|_{\ell_m^p}^2.
\end{align*}
To derive the last inequality, we used the sub-logarithmic growth of the sequence $c_k$ and the condition $0\leq s<\frac32-\frac1p$.  The term $\log^{\frac{p-2}{p}}(\langle k\rangle)$ in the second inequality above is only necessary when $2s\frac{p}{p-2}=1$, that is, $s=\frac12-\frac1p$.
\end{proof}

\begin{remark} It is important in the proof above that we are working with the sextic and higher order terms in $\beta$.  If we replaced $\beta_{\geq 6}$ by $\beta_{\geq 4}$, then the analogue of \eqref{1:41} would be
\begin{align}\label{1:61}
\Bigl\|c_k\langle k\rangle^{s}\sqrt{\beta_{\geq 4}\bigl(\tfrac12; u^k(t)\bigr)}\Bigr\|_{\ell_k^{p}}
&\lesssim \biggl\|\sum_{m\in \Z} c_k\langle k\rangle^s\tfrac{\log(\langle m-k\rangle)}{\langle m-k\rangle} \|\widehat u(t)\|_{L^2(I_m)}^2\biggr\|_{\ell_k^p}\notag\\
&\lesssim \biggl\|\sum_{|m|\geq \frac12|k|} \tfrac{\log(\langle m-k\rangle)}{\langle m-k\rangle} c_m^{2}\langle m\rangle^{2s} \|\widehat u(t)\|_{L^2(I_m)}^2\biggr\|_{\ell_k^p}\\
&\quad+ \biggl\|\sum_{|m|< \frac12|k|}c_k\langle k\rangle^{s -1} \log(\langle k\rangle) \|\widehat u(t)\|_{L^2(I_m)}^2\biggr\|_{\ell_k^p}.\notag
\end{align}
Thus, arguing as in the proof of Lemma~\ref{R6}, we encounter two changes: (i) We incur a logarithmic divergence for the first term on the right-hand side of \eqref{1:61} when $p=1$.
(ii) Bounding the second term on the right-hand side of \eqref{1:61} for $1< p<\infty$ requires the more stringent condition $0\leq s<1-\frac1p$.  This restriction would confine our analysis to the regime treated already in \cite{OW20}.

In particular, the arguments of Lemma~\ref{R6} would only yield
\begin{align}\label{1:81}
\Bigl\|c_k\langle k\rangle^{s}\sqrt{\beta_{\geq 4}\bigl(\tfrac12; u^k(t)\bigr)}\Bigr\|_{\ell_k^{p}}\lesssim \bigl\|c_k\langle k\rangle^{s}\|\widehat{u}(t,\xi)\|_{L_\xi^2(I_k)}\bigr\|_{\ell_k^p}^2
\end{align}
uniformly for $t\in I$, whenever $1<p<\infty$ and $0\leq s<1-\frac1p$.
\end{remark}

We now turn to the contribution to \eqref{3:33} of the quartic terms in $\beta$.  To control this for all $0\leq s<\frac32-\frac1p$, we rely on the explicit formula provided by Lemma~\ref{L:quartic}.

\begin{lemma}\label{L:beta4 small p}
Fix $1\leq p<\infty$ and $0\leq s<\frac32-\frac1p$.  Then 
\begin{align*}
\Bigl\|c_k\langle k\rangle^{s}\sqrt{\beta_4\bigl(\tfrac12; u^k(t)\bigr)}\Bigr\|_{\ell_k^{p}}\lesssim \bigl\|c_k\langle k\rangle^{s}\|\widehat{u}(t,\xi)\|_{L_\xi^2(I_k)}\bigr\|_{\ell_k^p}^2
\end{align*}
uniformly for $t\in I$.
\end{lemma}

\begin{proof}
In view of \eqref{1:81}, it suffices to prove the claim for (i) $p=1$ and $0\leq s<\frac12$; and (ii) $1<p<\infty$ and $1-\frac1p\leq s<\frac32-\frac1p$.

Using the explicit representation of the quartic term in the definition of $\alpha$ provided by Lemma~\ref{L:quartic}, together with H\"older and Young's convolution inequality, we may bound
\begin{align*}
\bigl|\beta_4\bigl(\tfrac12; u^k(t)\bigr)\bigr|
&\lesssim \Bigl\langle |\widehat  u(t, \xi)|, \tfrac{ |\widehat u(t, \xi)|}{\langle \xi-k\rangle} *\tfrac{ |\widehat u(t, \xi)|}{\langle \xi-k\rangle}* \tfrac{ |\widehat u(t, \xi)|}{\langle \xi-k\rangle^2}\Bigr\rangle_{L^2_\xi}\\
&\lesssim \|\widehat u(t, \xi)\|_{L^2_\xi} \bigl\| \tfrac{ |\widehat u(t, \xi)|}{\langle \xi-k\rangle}\bigr\|_{L^{\frac43}_\xi}^2\bigl\| \tfrac{ |\widehat u(t, \xi)|}{\langle \xi-k\rangle^2}\bigr\|_{L^1_\xi}\\
&\lesssim \|u(t)\|_{L^2} \bigl\| \tfrac{ |\widehat u(t, \xi)|}{\langle \xi-k\rangle}\bigr\|_{L^{\frac43}_\xi}^3\bigl\| \langle \xi-k\rangle^{-1}\|_{L^4_\xi}\\
&\lesssim \|u(t)\|_{L^2} \bigl\| \tfrac{ |\widehat u(t, \xi)|}{\langle \xi-k\rangle}\bigr\|_{L^{\frac43}_\xi}^3.
\end{align*}

By Lemma \ref{L:embedding} and the fact that $c_k\geq 1$ for all $k\in\mathbb{Z}$, we have 
$$
\|u(t)\|_{L^2}\lesssim \|u(t)\|_{M^{s,2}_p}\leq \bigl\|c_k\langle k\rangle^s\|{\widehat{u}(t,\xi)}\|_{L^2_\xi(I_k)}\bigr\|_{\ell_k^{p}}
$$
whenever $s\geq 0$ and $1\leq p\leq 2$ or $s>\frac12-\frac1p$ if $2<p<\infty$, which includes the range of parameters we are considering. It thus remains to show that
\begin{align}\label{2:34}
\Bigl\|c_k\langle k\rangle^{s}\bigl\| \tfrac{ |\widehat u(t, \xi)|}{\langle \xi-k\rangle}\bigr\|_{L^{\frac43}_\xi}^{\frac32}\Bigr\|_{\ell_k^{p}}\lesssim \bigl\|c_k\langle k\rangle^{s}\|\widehat{u}(t,\xi)\|_{L_\xi^2(I_k)}\bigr\|_{\ell_k^p}^{\frac32}.
\end{align}

Using the properties of the sequence $c_k$ and separating into the cases $2|m|\geq |k|$ and $2|m|<|k|$, we may bound
\begin{align}\label{2:35}
\Bigl\|c_k\langle k\rangle^{s}\bigl\| \tfrac{ |\widehat u(t, \xi)|}{\langle \xi-k\rangle}\bigr\|_{L^{\frac43}_\xi}^{\frac32}\Bigr\|_{\ell_k^{p}}
&\lesssim  \biggl\|c_k\langle k\rangle^s \Bigl\{\sum_{m\in \Z}\langle m-k\rangle^{-\frac43} \|\widehat u(t,\xi)\|_{L^{\frac43}_\xi(I_m)}^{\frac43}\Bigr\}^{\frac98}\biggr\|_{\ell_k^p}\notag\\
&\lesssim  \biggl\|\sum_{m\in \Z} c_k^{\frac89}\langle k\rangle^{\frac{8s}9} \langle m-k\rangle^{-\frac43} \|\widehat u(t)\|_{L^2(I_m)}^{\frac43}\biggr\|_{\ell_k^{\frac{9p}8}}^{\frac98}\notag\\
&\lesssim \biggl\|\sum_{|m|\geq \frac12|k|}\langle m-k\rangle^{-\frac43} c_m^{\frac43}\langle m\rangle^{\frac{4s}3} \|\widehat u(t)\|_{L^2(I_m)}^{\frac43}\biggr\|_{\ell_k^{\frac{9p}8}}^{\frac98}\\
&\quad+ \biggl\|\sum_{|m|< \frac12|k|}c_k^{\frac89}\langle k\rangle^{\frac{8s}9-\frac43}  \|\widehat u(t)\|_{L^2(I_m)}^{\frac43}\biggr\|_{\ell_k^{\frac{9p}8}}^{\frac98}.\notag
\end{align}

To estimate the terms on the right-hand side of \eqref{2:35}, we distinguish two cases:  If $1\leq p\leq \frac43$, we use the embedding $\ell^p\hookrightarrow \ell^{\frac43}$ to bound
\begin{align*}
\biggl\|\sum_{|m|\geq \frac12|k|}&\langle m-k\rangle^{-\frac43} c_m^{\frac43}\langle m\rangle^{\frac{4s}3} \|\widehat u(t)\|_{L^2(I_m)}^{\frac43}\biggr\|_{\ell_k^{\frac{9p}8}}\\
&\lesssim \sum_{m\in \Z} c_m^{\frac43}\langle m\rangle^{\frac{4s}3} \|\widehat u(t)\|_{L^2(I_m)}^{\frac43} \bigl\| \langle m-k\rangle^{-\frac43}\bigr\|_{\ell_k^{\frac{9p}8}}
\lesssim \bigl\|  c_m\langle m\rangle^{s} \|\widehat u(t)\|_{L^2(I_m)} \bigr\|_{\ell_m^p}^{\frac43}
\end{align*}
and
\begin{align*}
\biggl\|\sum_{|m|< \frac12|k|}  c_k^{\frac89}\langle k\rangle^{\frac{8s}9 -\frac43} \|\widehat u(t)\|_{L^2(I_m)}^{\frac43}\biggr\|_{\ell_k^{\frac{9p}8}}
&\lesssim  \sum_{m\in \Z}\|\widehat u(t)\|_{L^2(I_m)}^{\frac43} \bigl\| c_k^{\frac89}\langle k\rangle^{\frac{8s}9 -\frac43}\bigr\|_{\ell_k^{\frac{9p}8}}\\
&\lesssim \sum_{m\in \Z} c_m^{\frac43}\langle m\rangle^{\frac{4s}3} \|\widehat u(t)\|_{L^2(I_m)}^{\frac43}\\
&\lesssim  \bigl\| c_m\langle m\rangle^{s} \|\widehat u(t)\|_{L^2(I_m)} \bigr\|_{\ell_m^p}^{\frac43}.
\end{align*}
To obtain the second inequality in the last display we used that $ c_m^{\frac43}\langle m\rangle^{\frac{4s}3}  \geq 1$ for all $m\in \Z$ together with the fact $c_k^{\frac89}\langle k\rangle^{\frac{8s}9 -\frac43}\in \ell_k^{\frac{9p}8}$ whenever  $s<\frac32-\frac1p$.  This completes the proof of \eqref{2:34} for $1\leq p\leq \frac43$.  

Finally, we consider the contribution of the terms on the  right-hand side of \eqref{2:35} in the case $\frac43<p<\infty$.  Using Young's convolution inequality, we may bound
\begin{align*}
\biggl\|\sum_{|m|\geq \frac12|k|} & \langle m-k\rangle^{-\frac43} c_m^{\frac43}\langle m\rangle^{\frac{4s}3} \|\widehat u(t)\|_{L^2(I_m)}^{\frac43}\biggr\|_{\ell_k^{\frac{9p}8}}\\
&\lesssim \bigl\|\langle m\rangle^{-\frac43}\bigr\|_{\ell_m^{\frac{9p}{9p-4}}}\bigl\| c_m^{\frac43}\langle m\rangle^{\frac{4s}3} \|\widehat u(t)\|_{L^2(I_m)}^{\frac43} \bigr\|_{\ell_m^{\frac{3p}4}} 
\lesssim \bigl\|  c_m\langle m\rangle^{s} \|\widehat u(t)\|_{L^2(I_m)} \bigr\|_{\ell_m^p}^{\frac43},
\end{align*}
which is acceptable.  To control the second term on the  right-hand side of \eqref{2:35}, we employ the H\"older inequality to bound
\begin{align*}
\biggl\|&\sum_{|m|< \frac12|k|} c_k^{\frac89}\langle k\rangle^{\frac{8s}9 -\frac43} \|\widehat u(t)\|_{L^2(I_m)}^{\frac43}\biggr\|_{\ell_k^{\frac{9p}8}}\\
&\lesssim  \biggl\|c_k^{\frac89}\langle k\rangle^{\frac{8s}9 -\frac43} \bigl\|  c_m^{\frac43}\langle m\rangle^{\frac{4s}3} \|\widehat u(t)\|_{L^2(I_m)}^{\frac43}\bigr\|_{\ell_m^{\frac{3p}4}}\cdot  \|\langle m\rangle^{-\frac{4s}3}\|_{\ell_m^{\frac{3p}{3p-4}}(|m|< \frac12|k|)}\biggr\|_{\ell_k^{\frac{9p}8}}\\
&\lesssim  \bigl\|  c_m\langle m\rangle^{s} \|\widehat u(t)\|_{L^2(I_m)} \bigr\|_{\ell_m^p}^{\frac43}  \Bigl\|c_k^{\frac89}\langle k\rangle^{\frac{8s}9 -\frac43}\max\bigl\{ 1, \langle k\rangle^{-\frac{4s}3 +\frac{3p-4}{3p}} \log^{\frac{3p-4}{3p}}(\langle k\rangle)\bigr\} \Bigr\|_{\ell_k^{\frac{9p}8}}\\
&\lesssim  \bigl\|  c_m\langle m\rangle^{s} \|\widehat u(t)\|_{L^2(I_m)} \bigr\|_{\ell_m^p}^{\frac43},
\end{align*}
which is also acceptable.  In the computations above we used the sub-logarithmic growth of the sequence $c_k$ and the condition $0\leq s<\frac32-\frac1p$.  Note also that the term $\log^{\frac{3p-4}{3p}}(\langle k\rangle)$ is only necessary when $\frac{4s}3\cdot\frac{3p}{3p-4}=1$, that is,  when $s=\frac34-\frac1p$.
\end{proof}

Combining \eqref{3:33} with Lemmas~\ref{R6} and \ref{L:beta4 small p}, we arrive at
\begin{align}\label{boot}
\bigl\|c_k\langle k\rangle^s&\|\widehat{u}(t,\xi)\|_{L^2_\xi(I_k)}\bigr\|_{\ell_k^p}\notag\\
& \leq C_0\biggl\{ \bigl\|c_k\langle k\rangle^s\|\widehat{u}(0,\xi)\|_{L^2_\xi(I_k)}\bigr\|_{\ell_k^p} \Bigl[ 1 + \bigl\|c_k\langle k\rangle^s\|\widehat{u}(0,\xi)\|_{L^2_\xi(I_k)}\bigr\|_{\ell_k^p}^2\Bigr]\\
&\qquad\qquad+\bigl\|c_k\langle k\rangle^s\|\widehat{u}(t,\xi)\|_{L^2_\xi(I_k)}\bigr\|_{\ell_k^p}^2 +\bigl\|c_k\langle k\rangle^s\|\widehat{u}(t,\xi)\|_{L^2_\xi(I_k)}\bigr\|_{\ell_k^p}^3\biggr\}\notag
\end{align}
uniformly for $t\in I$, for all $1\leq p<\infty$ and $0\leq s<\frac32-\frac1p$.  Here $C_0>0$ is a fixed constant that incorporates the implicit constants appearing in \eqref{3:33} and Lemmas~\ref{R6} and \ref{L:beta4 small p}.  

We first discuss the question of the a priori bounds \eqref{apriori bds}. Taking $c_k\equiv 1$, recalling the small data assumption \eqref{small norm data}, and choosing $\varepsilon$ sufficiently small depending on $C_0$, a standard bootstrap argument yields
$$
\sup_{t\in I}\, \|u(t)\|_{M^{s,2}_p}\leq \tfrac32 C_0\varepsilon.
$$
Using the time-continuity of the \eqref{mkdv} and \eqref{nls} flows, this allows us to enlarge the interval $I$ on which \eqref{small norm} holds, which in turn implies \eqref{smalltail} on this larger interval.  Another straightforward continuity argument then yields 
\begin{align}\label{a priori bds}
\sup_{t\in \R} \, \|u(t)\|_{M^{s,2}_p}\leq \tfrac32 C_0\varepsilon \quad\text{for all $t\in \R$}.
\end{align}
This proves the desired a priori bounds for solutions to \eqref{mkdv} and \eqref{nls} satisfying the small data condition \eqref{small norm data}.

Now given an arbitrary (large) data $u(0)\in \mathcal S(\R)$, we may use Lemma~\ref{L:scaling} to reduce to the small data case.  Specifically, we may choose
$$
\lambda_0 \simeq \begin{cases}
[1+ \varepsilon^{-1} \|u(0)\|_{M^{s,2}_p}]^p &\text{ if } \ 2\leq p<\infty,\\[2mm]
[1+ \varepsilon^{-1} \|u(0)\|_{M^{s,2}_p}]^2 &\text{ if } \ 1\leq p\leq 2
\end{cases}
$$
to guarantee that the rescaled initial data satisfies $\|u_{\lambda_0}(0)\|_{M^{s,2}_p}\leq \varepsilon$.  The argument above then yields 
$$
\sup_{t\in \R} \|u_{\lambda_0}(t)\|_{M^{s,2}_p}\lesssim  \|u_{\lambda_0}(0)\|_{M^{s,2}_p}.
$$
Applying again Lemma~\ref{L:scaling} with $\lambda = \lambda_0^{-1}$ to reverse the scaling symmetry, we get
\begin{align}\label{a priori}
\sup_{t\in\R}\, \|u(t)\|_{M^{s,2}_p}\lesssim \bigl(1+\|u(0)\|_{M^{s,2}_p}\bigr)^{c(s,p)}\|u(0)\|_{M^{s,2}_p},
\end{align}
where $c(s,p)$ is as in Theorem~\ref{T:main}.

To complete the proof of Theorem~\ref{T:main}, it remains to show that orbits emanating from a set $Q\subset \mathcal S(\R)$, which is bounded and equicontinuous in $M^{s,2}_p(\R)$, form a family that is also bounded and equicontinuous in $M^{s,2}_p(\R)$.  Let
$$
Q^*:=\bigl\{u(t): t\in \R , \, u(0)\in Q, \, u\text{ solves either \eqref{mkdv} or \eqref{nls}}\bigr\} 
$$
denote the set of orbits.  The boundedness of $Q^*$ in $M^{s,2}_p(\R)$ follows from the a priori bounds \eqref{a priori} proven above.  It remains to demonstrate equicontinuity.

Using the scaling symmetries of \eqref{mkdv} and \eqref{nls} together with Lemma~\ref{L:scaling}, we may rescale the data in $Q$ to ensure that the rescaled functions satisfy the small data condition \eqref{small norm data}.  Specifically, taking
$$
\lambda_0 \simeq \begin{cases}
[1+ \varepsilon^{-1} \sup_{u\in Q}\, \|u\|_{M^{s,2}_p}]^p &\text{ if } \  2\leq p<\infty,\\[2mm]
[1+ \varepsilon^{-1} \sup_{u\in Q}\,\|u\|_{M^{s,2}_p}]^2 &\text{ if } \ 1\leq p\leq 2
\end{cases}
$$
we can guarantee that 
$$
\sup_{u\in Q}\, \|u_{\lambda_0}\|_{M^{s,2}_p} \leq \tfrac12\varepsilon.
$$
As noted in subsection~\ref{SS:tss}, scaling preserves the equicontinuity property; in particular, the set $Q_{\lambda_0}:=\{ u_{\lambda_0}:\, u\in Q\}$ is equicontinuous in $M^{s,2}_p(\R)$.  Using the characterization of equicontinuity provided by Proposition~\ref{P:equi} for the family $Q_{\lambda_0}$, we may find a sequence $\{c_k\}_{k\in \Z}$ satisfying properties (i)---(v) listed therein.  In particular,
$$
\sup_{u\in Q_{\lambda_0}}\, \bigl\|c_k\langle k\rangle^s\|\widehat u\|_{L^2(I_k)}\bigr\|_{\ell_k^p}\leq \varepsilon.
$$
Inserting this into \eqref{boot} and employing the same continuity arguments used to establish the a priori bounds above, we arrive at
$$
 \sup_{u\in Q_{\lambda_0}^*}\, \bigl\|c_k\langle k\rangle^s\|\widehat u\|_{L^2(I_k)}\bigr\|_{\ell_k^p}\leq \tfrac32 C_0\varepsilon,
$$
where 
$$
Q_{\lambda_0}^*:=\bigl\{u(t): t\in \R , \, u(0)\in Q_{\lambda_0}, \, u\text{ solves either \eqref{mkdv} or \eqref{nls}}\bigr\}.
$$
By Proposition~\ref{P:equi}, this implies that the set $Q_{\lambda_0}^*$ is equicontinuous in $M^{s,2}_p(\R)$.  Undoing the scaling symmetry (which preserves the equicontinuity property), we derive that the set $Q^*$ is equicontinuous in $M^{s,2}_p(\R)$, as desired.

\section{Global well-posedness in modulation spaces}
\label{well posed}

In this section we prove Theorem~\ref{T:Main}, using crucially the well-posedness  result proved in \cite{HKV}.

Both \eqref{nls} and \eqref{mkdv} have long been known to be globally well-posed for Schwartz initial data $u_0$.  This allows us to define a jointly continuous data-to-solution map
$\Phi:\R\times\mathcal S(\R) \to \mathcal S(\R)$.  The basic question we wish to understand is this: Does this map extend continuously to $M^{s,2}_p(\R)$?   Evidently, any such extension is automatically unique because of the density of Schwartz functions in the space $M^{s,2}_p(\R)$.

The paper \cite{HKV} proved that such a continuous extension $\Phi:\R\times H^\sigma(\R) \to H^\sigma(\R)$ is possible for any $\sigma>-\frac12$ for both the \eqref{mkdv} and \eqref{nls} equations.  Prior to this, \eqref{mkdv} required $\sigma\geq \frac14$ and \eqref{nls} required $\sigma\geq 0$; see \cite{Guo, Kishimoto, MR915266}. We will need the full range $\sigma>-\frac12$ in order to treat $s=0$ and $p$ large, for example.  The crux of the argument in \cite{HKV}  is the following:

\begin{theorem}[\cite{HKV}]\label{T1}
Fix $\sigma>-\tfrac12$ and $T>0$.  For any sequence $u_n(0,x)\in \mathcal S(\R)$ of initial data that is $H^\sigma$-Cauchy, the corresponding sequence of solutions $u_n(t,x)$ to \eqref{mkdv}, or to \eqref{nls}, is Cauchy in $C_t([-T,T],H^\sigma(\R))$.
\end{theorem}

Next, we will prove an analogue of this assertion for the spaces $M^{s,2}_p(\R)$; see Theorem~\ref{T:CMX} below.  To achieve this, we will rely on the full strength of Theorem~\ref{T:main} and Proposition~\ref{P:upgrade convg}.  After that, we will complete the proof of Theorem~\ref{T:Main} by demonstrating that Theorem~\ref{T:CMX} guarantees that the data-to-solution map can be uniquely and continuously extended to  $M^{s,2}_p(\R)$.

\begin{theorem}\label{T:CMX}
Fix $1\leq p<\infty$, $0\leq s<\frac32-\frac1p$, and $T>0$.   For any sequence $u_n(0,x)\in \mathcal S(\R)$ of initial data that is $M^{s,2}_p(\R)$-Cauchy, the corresponding sequence of solutions $u_n(t,x)$ to \eqref{mkdv}, or to \eqref{nls}, is Cauchy in $C_t([-T,T],M^{s,2}_p(\R))$.
\end{theorem}

\begin{proof}
Let $\sigma=\sigma(p,s)>-\tfrac12$ be chosen satisfying \eqref{262}.  Lemma~\ref{L:embedding} then guarantees that $M^{s,2}_p(\R)\hookrightarrow H^\sigma(\R)$ and consequently, $u_n(0,x)$ is  $H^\sigma$-Cauchy.  Applying Theorem~\ref{T1}, we deduce that the corresponding solutions converge to some $u$ in $C_t([-T,T],H^\sigma(\R))$.

As $u_n(0)$ converge in $M^{s,2}_p(\R)$, we know that the set $\{ u_n(0) \}$ is not only $M^{s,2}_p(\R)$-bounded, but also that it is equicontinuous (as noted in Remark~\ref{R:equi}).  In view of Theorem~\ref{T:main}, this implies that the set of orbits $\{u_n(t):\, t\in \R, \, n\geq 1\}$ is bounded and equicontinuous in $M^{s,2}_p(\R)$.

Proposition~\ref{P:upgrade convg} then allows us to upgrade the convergence of $u_n$ in the space  $C_t([-T,T],H^\sigma(\R))$ to convergence in $C_t([-T,T],M^{s,2}_p(\R))$.
\end{proof}

We are now in a position to complete the

\begin{proof}[Proof of Theorem~\ref{T:Main}]
Given initial data $u_0\in M^{s,2}_p(\R)$, we define $\Phi(t,u_0)$ as follows:  Let $u_n(0)$ be a sequence of Schwartz functions that converges to $u_0$ in $M^{s,2}_p(\R)$.  By Theorem~\ref{T:CMX}, for any $T>0$, the sequence of solutions $u_n(t)$ converges in $C_t([-T,T], M^{s,2}_p(\R))$.  Consequently, we may define
\begin{equation*}
\Phi(t,u_0) := \lim_{n\to\infty} u_n(t) .
\end{equation*}
Theorem~\ref{T:CMX} also guarantees that this limit is independent of the choice of sequence $u_n(0)$.
	
It remains to show that $\Phi$ is jointly continuous.  We achieve this in two steps.  First, we demonstrate continuity in time.  Specifically, given $u_0\in M^{s,2}_p(\R)$ and a sequence $t_n\to t_0\in \R$, we claim that
\begin{align}\label{cont in time}
\lim_{n\to \infty}\| \Phi(t_n,u_0)-\Phi(t_0,u_0)\|_{M^{s,2}_p}=0.
\end{align}
This follows from the continuity in time of Schwartz solutions.  Indeed, setting $T= \sup_{n\geq 1} |t_n|$ and $\varepsilon >0$, by the definition of $\Phi$ there exists a Schwartz solution $\widetilde u$ so that
$$
\sup_{|t|\leq T}\, \| \Phi(t,u_0)-\widetilde u(t)\|_{M^{s,2}_p}\leq \varepsilon.
$$
By the triangle inequality, we may then bound
\begin{align*}
\|&\Phi(t_n,u_0) -\Phi(t_0,u_0)\|_{M^{s,2}_p} \\
&\leq  \| \Phi(t_n,u_0)-\widetilde u(t_n)\|_{M^{s,2}_p}+ \| \Phi(t_0,u_0)-\widetilde u(t_0)\|_{M^{s,2}_p} + \|\widetilde u(t_n)-\widetilde u(t_0)\|_{M^{s,2}_p} \\
&\leq 2\sup_{|t|\leq T} \, \| \Phi(t,u_0)-\widetilde u(t)\|_{M^{s,2}_p} +  \|\widetilde u(t_n)-\widetilde u(t_0)\|_{M^{s,2}_p} \\
&\leq 2\varepsilon + \|\widetilde u(t_n)-\widetilde u(t_0)\|_{M^{s,2}_p}.
\end{align*}
Letting $n\to \infty$ and recalling that $\varepsilon>0$ was arbitrary, we deduce \eqref{cont in time}.

Turning to joint continuity, we consider a sequence of initial data $u_{0,n} \in M^{s,2}_p(\R)$ that converges to $u_0$ in $M^{s,2}_p(\R)$ and a sequence of times $t_n\to t_0\in \R$.  Let $T= \sup_{n\geq 1} |t_n|$.  By the definition of $\Phi$, we may choose another sequence $\widetilde u_n (t)$ of Schwartz solutions to \eqref{mkdv}/\eqref{nls} such that
\begin{equation}\label{11:47}
\lim_{n\to \infty}\,\sup_{|t|\leq T}  \, \|\Phi(t,u_{0,n}) -\widetilde u_n (t) \|_{M^{s,2}_p}=0.
\end{equation}
In particular, $\widetilde u_n (0) \to u_0$ in $M^{s,2}_p(\R)$, and so Theorem~\ref{T:CMX} yields
\begin{equation}\label{11:48}
\lim_{n\to \infty}\,\sup_{|t|\leq T}  \, \| \widetilde u_n (t) - \Phi(t,u_0) \|_{M^{s,2}_p} = 0.
\end{equation}
By the triangle inequality, we may then bound
\begin{align*}
\|\Phi(t_n,u_{0,n})& -\Phi(t_0,u_0)\|_{M^{s,2}_p} \\
&\leq  \| \Phi(t_n,u_{0,n})-\widetilde u_n(t_n)\|_{M^{s,2}_p}+ \| \widetilde u_n(t_n)-\Phi(t_n,u_0)\|_{M^{s,2}_p} \\
&\quad+\| \Phi(t_n,u_0)-\Phi(t_0,u_0)\|_{M^{s,2}_p}\\
&\leq \sup_{|t|\leq T}  \, \|\Phi(t,u_{0,n}) -\widetilde u_n (t) \|_{M^{s,2}_p} +\sup_{|t|\leq T}  \, \| \widetilde u_n (t) - \Phi(t,u_0) \|_{M^{s,2}_p} \\
&\quad +\| \Phi(t_n,u_0)-\Phi(t_0,u_0)\|_{M^{s,2}_p}.
\end{align*}
The right-hand side above converges to zero as $n\to\infty$ by  \eqref{11:47}, \eqref{11:48}, and  \eqref{cont in time}; this demonstrates that $\Phi$ is jointly continuous.
\end{proof}

\end{document}